\documentclass[12pt,a4paper]{article}
\usepackage[utf8]{inputenc}
\usepackage{indentfirst} 
\usepackage{amsfonts}
\usepackage{amssymb}
\usepackage{amsmath}
\usepackage{amsthm}
\usepackage{latexsym}

\newcommand{\wfont}[1]{\textit{#1}}
\newcommand{\ffont}[1]{\textsf{#1}}
\newcommand{\mfont}[1]{\mathfrak{#1}}
\newcommand{\sfont}[1]{\textbf{\textit{#1}}}
\newcommand{\tfont}[1]{\textbf{#1}}

\newcommand{\romb}[1]{ \langle #1 \rangle }
\newcommand{\sprv}{\begin{proof}}
\newcommand{\eprv}{\end{proof}}

\newcommand{\word}{\sfont{W}}

\newcommand{\defiff}{\stackrel{\mathrm{def}}{\iff}}
\newcommand{\wore}{\precsim}
\newcommand{\wor}{\prec}
\newcommand{\NF}{\sfont{NF}}
\newcommand{\PA}{\tfont{PA}}

\newcommand{\GLP}{\tfont{GLP}}
\newcommand{\ZF}{\tfont{ZF}}
\newcommand{\On}{\sfont{On}}
\newcommand{\multi}{\mathcal{P}^{<\omega}_{\mathrm{multi}}}
\newcommand{\pfin}{\mathcal{P}^{<\omega}}
\newcommand{\Lim}{\sfont{Lim}}

\newcommand{\dom}{\mathrm{dom}}
\newcommand{\ran}{\mathrm{ran}}

\newcommand{\Th}{\mathbf{Th}}
\newcommand{\refl}[1]{\mathcal{R}_{#1}}
\newcommand{\WSExt}[1]{#1'}
\newcommand{\WMSExt}[1]{#1''}
\newcommand{\HFExt}[1]{#1^+}

\newtheorem{theorem}{Theory}
\newtheorem{fact}{Fact}
\newtheorem{remark}{Remark}
\newtheorem{lemma}{Lemma}
\newtheorem{corollary}{Corollary}
\newtheorem{proposition}{Proposition}

\begin{document}

\author{Fedor~Pakhomov\thanks{This work was partially supported by RFFI grant 12-01-00888\_a and Dynasty foundation.}\\Steklov Mathematical Institute,\\Moscow\\ \texttt{pakhfn@mi.ras.ru}}
\date{December 2013}
\title{On Elementary Theories of Ordinal Notation Systems based on Reflection Principles}
\maketitle

\begin{abstract} We consider the constructive ordinal notation system for the ordinal  $\varepsilon_0$ that were introduced by L.D.~Beklemishev. There are fragments of this system that are ordinal notation systems for the smaller ordinals $\omega_n$ (towers of $\omega$-exponentiations of the height $n$).  This systems are based on Japaridze's provability logic  $\tfont{GLP}$.  They are closely related with the technique of ordinal analysis of $\tfont{PA}$ and fragments of $\tfont{PA}$ based on iterated reflection principles. We consider this  notation system and it's fragments as structures with the signatures selected in a natural way. We prove that the full notation system and it's fragments, for ordinals  $\ge\omega_4$, have undecidable elementary theories. We also prove that the fragments of the full system, for ordinals $\le\omega_3$, have decidable elementary theories. We obtain some results about decidability of elementary theory, for the ordinal notation systems with  weaker signatures. 
\end{abstract}

\section{Introduction} The problems of calculation of the proof-theoretic ordinal of a theory are well-known in proof theory. G.~Gentzen was the pioneer in this field \cite{Gen36}; there is an overview on this subject by M.~Rathjen \cite{Rat99}. 

 Proof-theoretic ordinals of theories normally are calculated in the terms of constructive ordinal notation systems. The general theory of such a systems is due to A.~Church and S.C.~Kleene\cite{Kln38}\cite{Chr38}. The classical method to encode ordinal notation systems is Kleene $\mathcal{O}$\cite{Kln38}. The ordinal analysis usually involve the ordinal notation systems in another form; we describe the typical kind of systems that are used in ordinal analysis. Some functions $f_0,f_1,\ldots$ from ordinals to ordinals are considered. These functions may have different arity and some of them are $0$-ary functions, i.e. constants. The set  $\sfont{T}$ of all closed terms built of  $f_0,f_1,\ldots$ is considered. There is the binary predicate $<_v$ that compares the values of the terms from $\sfont{T}$. For systems that are normally considered, the predicate $<_v$ is computable by a simple algorithm. An ordinal $\alpha$ is such that, for every ordinal $\beta<\alpha$, the ordinal $\beta$ is equal to the value of some term from $\sfont{T}$. From the recursiveness of $<_v$ it follows that the term value equality predicate $=_v$ is recursive too. And also, it follows that the predicate $\ffont{P}_\alpha(x)$ $$\ffont{P}_\alpha(t)\defiff \beta<\alpha\mbox{, where $\beta$ is the value of a term $t$}$$ is recursive.  Hence the recursive structure $(\{t\in \sfont{T}\mid \ffont{P}_\alpha(t)\}/{=_v},<_v)$ is isomorphic to $(\alpha,<)$. We consider the ordinal notation system as the recursive structure  $(\sfont{T}/{=_v},<_v,f_0,f_1,\ldots)$. 

In the present paper we consider the decidability of the elementary theory problem for some ordinal notation systems.

For ordinals without additional structure the decidability of elementary theory problem were studied by  A.~Tarski and A.~Mostowski \cite{TarsMos49}\cite{DMT78}. It were shown that, for every ordinal $\alpha$, the elementary theory $\Th(\alpha,<)$ is decidable. Later this result was strengthen by  U.R.~B\"uchi. He had shown that, for every ordinal $\alpha$, the weak monadic theory of the structure $(\alpha,<)$ is decidable \cite{Buch65}. He also had constructed an interpretation of the elementary theory  $\Th(2^\alpha,<,+)$ in the weak monadic theory of $(\alpha,<)$. Thus he had shown that the first is decidable.

The ordinal $\varepsilon_0$ is the proof-theoretic ordinal of $\PA$ \cite{Gen36}. A cofinal sequence for an ordinal $\alpha$ is a sequence of ordinals $\beta_0,\beta_1,\ldots$ such that every $\beta_i<\alpha$ and  $\sup\{\beta_i\mid i\in \omega\}=\alpha$. There is the standard choice of cofinal sequences for the ordinals less than $\varepsilon_0$.  L.~Braud \cite{Brau09} had proved the decidability of the weak monadic theory of  $(\alpha,<,\ffont{Cs})$, where  $\alpha$ is some ordinal less than $\varepsilon_0$ and $\ffont{Cs}(x,y)$ is the predicate $$ \ffont{Cs}(\beta,\gamma) \defiff \mbox{ $\gamma$ is a member of the standard cofinal sequence for $\beta$}.$$ 

There are several different ``natural'' ordinal notation systems for the ordinals  below  $\varepsilon_0$ \cite{Lee07}. One of them were introduced by L.D.~Beklemishev \cite{Bek04-1}; we give it in the form that is slightly different from the form from \cite{Bek04-1}. There is a set $\word_\omega$, an equivalence relation $\sim$ on $\word_\omega$, and a binary relation $\wor$ on $\word_\omega$ such that $\wor$ is compatible with $\sim$; $(\varepsilon_0,<)$ and $(\word_\omega/{\sim},\wor)$ are isomorphic. There is a constant $\varLambda\in\word_\omega$ and functions $a_i\colon \word_\omega\to \word_\omega$, for every number $i$. Functions $a_i$ are compatible with $\sim$.  Every element of $\word_\omega$ is the value of the unique closed term built of $\varLambda,a_0,a_1,\ldots$.  Structure $(\word_\omega/{\sim},\wor,\varLambda,a_0,a_1,\ldots)$ is an ordinal notation system up to  $\varepsilon_0$. For every $n$, we denote by $\word_n$ the set of the values of all terms built of $\varLambda,a_0,\ldots,a_n$. The structures  $(\word_n/{\sim},\wor,\varLambda,a_0,\ldots,a_n)$ are ordinal notation systems for the smaller ordinals $\omega_{n+1}$. Here ordinals $\omega_n$ are defined as the following:
\begin{enumerate}
\item $\omega_0=1$;
\item $\omega_{n+1}=\omega^{\omega_n}$;
\item $\omega_\omega=\varepsilon_0=\lim \limits_{n\to \omega} \omega_n$.
\end{enumerate}

We prove that the elementary theory  $\Th(\word_\omega/{\sim},\wor,\varLambda,a_0,a_1,\ldots)$ is undecidable. For every ordinal $\alpha\in[3,\omega]$, we prove that the elementary theory $\Th(\word_\alpha/{\sim},\wor,a_1,a_3)$ is undecidable. Also, for every $\alpha\in[2,\omega]$, we show that the elementary theory $\Th(\word_\alpha/{\sim},\wor,\varLambda,a_0,a_1,a_2)$ is decidable. 

There is a natural binary operation $\land\colon \word_\omega\times \word_\omega\to\word_\omega$; $\land$ is compatible with $\sim$. For every $\alpha\le\omega$, the set $\word_\alpha$ is closed under $\land$. In \cite{Pak12} it was shown  that the elementary theory $\Th(\word_\alpha/{\sim},\land)$ is undecidable, for every $\alpha\in[2,\omega]$. In that paper it was also proved that the elementary theory $\Th(\word_\alpha/{\sim},\land)$ is decidable, for every $\alpha\in\{0,1\}$. There were shown that, for every $\alpha\le \omega$, the relation  $\wor$ and functions $a_i$ are definable in the structure $(\word_\alpha/{\sim},\land)$. In the present paper we consider structures with the same domains as in the paper \cite{Pak12} but our signatures have less definability power than the signatures from \cite{Pak12}. This weakening have impact on decidability-undecidability border for $\alpha$. The elementary theory $\Th(\word_2/{\sim},\land)$ is undecidable, but the elementary theory $\Th(\word_2/{\sim},\wor,\varLambda,a_0,a_1,a_2)$ is decidable.


\subsection{Ordinal analysis of $\PA$ by iterated reflection principles}
In the subsection we briefly describe the origin of the ordinal notation system under consideration (there are more information on the subject in \cite{Bek03}, \cite{Bek04-1}, \cite{Bek05}). 

We consider recursively axiomatizable theories in the language of the first order arithmetic $(0,S,+,\cdot)$ as algorithms enumerating non-logical axioms. It is well-known that one can formally work with recursively axiomatizable theories within powerful enough arithmetic theories.

There are classes of arithmetical formulas  $\Sigma_n$. For a number $n$, the class $\Sigma_n$ consists of all formulas of the form $$\exists x_1\ldots \exists x_{m_{1}}\forall x_{m_{1}+1}\ldots \forall x_{m_2}\ldots \mathrm{Q}_{n} x_{m_{n-1}}\ldots \mathrm{Q}_{n} x_{m_{n}}\ffont{A},$$ where $\ffont{A}$ is a formula with bounded quantifiers, $\mathrm{Q}_{n}=\forall$, if  $n$ is even and $\mathrm{Q}_{n}=\exists$, if $n$ is odd. There are formulas $\ffont{RFN}_{\Sigma_n}(x)$ such that, for every number $n$ and arithmetic recursively axiomatizable theory $\tfont{T}$, the proposition $\ffont{RFN}_{\Sigma_n}(\tfont{T})$ means ``for every formula $\ffont{A}(x)\in\Sigma_n$, if $\tfont{T}$ proves $\ffont{A}(k)$, for every individual number $k$, then $\forall x \ffont{A}(x)$ is true.'' We note that $\ffont{RFN}_{\Sigma_0}(\tfont{U})$ is equivalent to a proposition that means ``$\tfont{U}$ is consistent.'' 

There is a relation on arithmetic recursively axiomatizable theories $<_{Con}$: 
$$\tfont{U}_1<_{Con}\tfont{U}_2 \defiff \tfont{U}_2\vdash \ffont{RFN}_{\Sigma_0}(\tfont{U}_1).$$
We consider suitable subtheory $\tfont{T}_0$ of Peano Arithmetic $\PA$; we choose $\tfont{T}_0=\tfont{I}\boldsymbol\Delta_0+\ffont{Exp}$ (there is a definition of this theory in \cite{HajPud98}), to be precise. We give operations $\refl{0}$,$\refl{1}$,$\ldots$ on arithmetic recursively axiomatizable theories: $$\refl{n}\colon \tfont{U} \longmapsto \tfont{T}_0+ \ffont{RFN}_{\Sigma_n}(\tfont{U}).$$ 

If $\tfont{T}$ and $\tfont{U}$ are arithmetic recursively axiomatizable theories with equal sets of theorems, then we write $\tfont{T}\equiv\tfont{U}$.

We consider the set of arithmetic recursively axiomatizable theories $\mathfrak{S}_\omega$; $\mathfrak{S}_{\omega}$ is the closure of $\{\tfont{T}_0\}$ under the application of all $\refl{k}$. Note that $(\mathfrak{S}_\omega,\tfont{T}_0,\equiv,<_{Con},\refl{0},\refl{1},\ldots)$ is isomorphic to combinatorially defined structure $(\word_{\alpha},\varLambda,\sim,\wor,a_0,a_1,\ldots)$; we will define the later structure in  the next section. We call elements of $\word_{\omega}$ and $\mathfrak{S}_\omega$ corresponding, if they are the images of each other under the isomorphism.

We consider he theory that is axiomatizable by all axioms of theories from $\mathfrak{S}_\omega$. That theory is just an alternative axiomatization of $\PA$. By a transfinite induction on $(\word_{\omega},\wor)$ it can be proved that the theories from  $\mathfrak{S}_\omega$ are consistent. From the later it follows that $\PA$ is consistent. In fact the the step of the transfinite induction can be proved in the weak subtheory of $\PA$. Thus $(\word_{\omega}/{\sim},\varLambda,\wor,a_0,a_1,\ldots)$ is an ordinal notation system up to $\varepsilon_0$ that is extracted directly from the described proof of the consistency of $\PA$. 


\section{Ordinal notation system}
In the section we give a new combinatorial definition of the ordinal notation system we are interested in. Note that early this system were considered in the context of Japaridze's provability logic $\GLP$ \cite{Bek03}\cite{Bek04-2}.  The equivalency of the new definition with the older one can be proved using several propositions from \cite{Bek04-2}; essentially, we show that in Fact \ref{word_theory_correspondence}.

We denote by $\word_{\omega}$ the set of all strings over the alphabet of all natural numbers $0,1,\ldots$. We call elements of $\word_{\omega}$  {\it words}. We denote words by symbols $\wfont{A},\wfont{B},\wfont{C},\wfont{D}$,$\ldots$. For all $\wfont{A},\wfont{B}\in \word_{\omega}$ we denote by $\wfont{A}\wfont{B}$ the concatenation of $\wfont{A}$ and $\wfont{B}$. For a word $\wfont{A}\in \word_{\omega}$ and a number $n\in \omega$ we denote by $\wfont{A}^n$ the word $\underbrace{\wfont{A}\wfont{A}\ldots\wfont{A}}_{\mbox{$n$ times}}$. We denote by $\varLambda$ the empty word. We denote  by $|\wfont{A}|$ the length of $\wfont{A}$.

For every  $k\in \omega$, we denote by $\sfont{S}_k$ the set of all words $\wfont{A}$ from $\word_{\omega}$ such that all symbols from $\wfont{A}$ are $\ge k$. For $\alpha\le \omega$, we denote by $\word_\alpha$ the set of all words $\wfont{A}$ from $\word_{\omega}$ such that all symbols from $\wfont{A}$ are $\le \alpha$.

We start the definition of the preorder $\wore$ on $\word_{\omega}$. In the terms of $\wore$ we give an equivalence relation $\sim$ and binary relation $\wor$: 
$$\wfont{A}\sim \wfont{B}\defiff \wfont{A} \wore \wfont{B} \& \wfont{B}\wore \wfont{A},$$
$$\wfont{A}\wor \wfont{B}\defiff \wfont{A} \wore \wfont{B} \& \lnot \wfont{B}\wore \wfont{A}.$$
Further without any additional comments we use $\wore$ as the standard  preorder on $\word_{\omega}$. The previous sentence apply to  notions related to some comparing, i.e. ``the minimal element of a set $\sfont{A}\subseteq \word_{\omega}$'',  ``a word $\wfont{A}$ is less (greater, not less, not greater) than a word $\wfont{B}$'', etc. We say that a sequence $(\wfont{A}_1,\ldots,\wfont{A}_n)$ of elements of $\word_{\omega}$  is lexicographic not greater than a sequence $(\wfont{B}_1,\ldots \wfont{B}_m)$ of elements of $\word_{\omega}$ iff either $n\le m$ and $\wfont{A}_i\sim \wfont{B}_i$ or there exists $s<\min(m,n)$ such that, for numbers $i$ from $1$ to $s$, we have $\wfont{A}_i \sim \wfont{B}_i$  and $\wfont{A}_{s+1}\wore\wfont{B}_{s+1}$. Note that if $\wore$ is a linear preorder on a set $\sfont{A}\subseteq \word_{\omega}$, then the lexicographical comparison on the set $\sfont{A}^{<\omega}$  of all sequences with elements from $\sfont{A}$ is a linear preorder.   

By definition we put $\varLambda\wore \varLambda$. 

Suppose $r$ is a natural number and  $\wore$-comparisons are defined for all pairs $(\wfont{A}',\wfont{B}')$ such that, for some $n$, the word $\wfont{A}'\wfont{B}'$ lies in $\sfont{S}_{n}\cap \word_{n+r-1}$. Let us determine the $\wore$-comparison for all pairs $(\wfont{A},\wfont{B})$ such that $\wfont{A}\wfont{B}\in\sfont{S}_{n}\cap \word_{n+r}$, for some $n$. We consider pair $(\wfont{A},\wfont{B})$ such that $\wfont{A}\wfont{B}\in\sfont{S}_{n}\cap \word_{n+r}$, where $n$ is the minimal symbol from $\wfont{A}\wfont{B}$. Obviously, we can find the unique number $k$, words $\wfont{A}_1,\ldots,\wfont{A}_k\in \sfont{S}_{n+1}\cap \word_{n+r}$, natural number $l$ and words $\wfont{B}_1,\ldots,\wfont{B}_l\in \sfont{S}_{n+1}\cap \word_{n+r}$ such that $\wfont{A}=\wfont{A}_1n \ldots n \wfont{A}_k$ and $\wfont{B}=\wfont{B}_1n \ldots n\wfont{B}_l$. Note that we have all pairwise  $\wore$-comparison between elements of $\{\wfont{A}_1,\ldots,\wfont{A}_k,\wfont{B}_1,\ldots,\wfont{B}_l\}$. Suppose $(\wfont{C}_1,\ldots,\wfont{C}_f)$ and $(\wfont{D}_1,\ldots,\wfont{D}_g)$ are lexicographically maximal subsequences of $(\wfont{A}_1,\ldots,\wfont{A}_l)$ and $(\wfont{B}_1,\ldots,\wfont{B}_k)$, respectively. We give the $\wore$-comparison of $\wfont{A}$ and $\wfont{B}$ as the lexicographical comparison of the sequences $(\wfont{C}_1,\ldots,\wfont{C}_f)$ and $(\wfont{D}_1,\ldots,\wfont{D}_g)$. 

By simultaneous induction on $r$ we prove the two following propositions, for all $r$:
\begin{enumerate}
\item for all $n$, the binary relation $\wore$ is a linear preorder on the set $\sfont{S}_n\cap \word_{n+r}$;
\item Remark \ref{wore_characterization}, for the case of $\wfont{A}_1,\ldots,\wfont{A}_k\in \sfont{S}_{n+1}\cap \word_{n+r}$ and $\wfont{B}_1,\ldots,\wfont{B}_l\in \sfont{S}_{n+1}\cap \word_{n+r}$.
\end{enumerate}

\begin{remark} \label{wore_characterization} Suppose we have a natural number $n$, words $\wfont{A}_1,\ldots,\wfont{A}_k\in \sfont{S}_{n+1}$, and words $\wfont{B}_1,\ldots,\wfont{B}_l\in \sfont{S}_{n+1}$. And suppose $(\wfont{C}_1,\ldots,\wfont{C}_f)$ and $(\wfont{D}_1,\ldots,\wfont{D}_g)$ are lexicographically maximal subsequences of $(\wfont{A}_1,\ldots,\wfont{A}_k)$ and $(\wfont{B}_1,\ldots,\wfont{B}_l)$, respectively. Then $\wfont{A}_1n\ldots n\wfont{A}_k\wore \wfont{B}_1 n \ldots n\wfont{B}_l$ iff $(\wfont{C}_1,\ldots,\wfont{C}_f)$ is lexicographically not greater than $(\wfont{D}_1,\ldots,\wfont{D}_g)$.   
\end{remark}

Thus $\wore$ is a linear preorder on $\word_{\omega}$. 

\begin{fact} \label{word_theory_correspondence} For all $n_1,\ldots, n_k$, $m_1,\ldots,m_l$ we have the following equivalences:
\begin{enumerate} 
\item$n_1\ldots n_k\wor m_1 \ldots m_l \iff $ 

$\refl{n_k}(\ldots (\refl{n_1}(\tfont{T}_0))\ldots)<_{Con}\refl{m_l}(\ldots(\refl{m_1}(\tfont{T}_0))\ldots);$
\item$n_1\ldots n_k\sim m_1 \ldots m_l \iff$

$\refl{n_k}(\ldots (\refl{n_1}(\tfont{T}_0))\ldots)\equiv\refl{m_l}(\ldots(\refl{m_1}(\tfont{T}_0))\ldots).$
\end{enumerate}
\end{fact}
\begin{proof} Essentially, we prove that the ordinal notation system that we have defined is equivalent to the ordinal notation system from \cite{Bek04-1}\cite{Bek04-2}. System from \cite{Bek04-1}\cite{Bek04-2} is based on Japaridze's provability logic  $\tfont{GLP}$. We don't give a definition of the logic $\tfont{GLP}$ here, in this proof we assume that a reader is familiar with the logic $\tfont{GLP}$.

 Suppose $\wfont{A}=n_1\ldots n_k$ is a word from $\word_{\omega}$. We denote by $\wfont{A}^{\star}$ the theory $\refl{n_k}(\ldots (\refl{n_1}(\tfont{T}_0))\ldots)$. We denote by $\wfont{A}^{\#}$ the polymodal formula $\romb{n_k}\ldots \romb{n_1}\top$. For polymodal formulas  $\varphi$ and $\psi$, we denote by $\varphi<_0\psi$ the formula $\psi\to\romb{0}\varphi$.

As far as the author knows, it is unknown whether $\tfont{GLP}$ is complete with respect to arithmetical semantics with the basis theory $\tfont{T}_0=\tfont{I}\boldsymbol\Delta_0+\ffont{Exp}$. We prove the completeness for the specific class of formulas. Let us show that for an arbitrary words $\wfont{A},\wfont{B}\in\word_{\omega}$ we have the following: \begin{enumerate}\item$\tfont{GLP}\vdash \wfont{A}^{\#}\;\leftrightarrow\;\wfont{B}^{\#}\iff \wfont{A}^{\star}\equiv\wfont{B}^{\star}$, \item$\tfont{GLP}\vdash \wfont{A}^{\#}<_0 \wfont{B}^{\#}\iff \wfont{A}^{\star}<_{Con}\wfont{B}^{\star}$.\end{enumerate}  
Both $\Rightarrow$ implications here follows from the arithmetic correctness for the logic $\tfont{GLP}$ \cite[Lemma~5.3]{Bek05}. The reverse implications $\Leftarrow$ holds, because
\begin{enumerate}
\item from \cite[Proposition~3]{Bek04-2} and \cite[Proposition~4]{Bek04-2} it follows that at least one of the following propositions holds:
\begin{enumerate}
\item $\tfont{GLP}\vdash \wfont{A}^{\#}<_0\wfont{B}^{\#}$,
\item $\tfont{GLP}\vdash \wfont{A}^{\#}\;\leftrightarrow\;\wfont{B}^{\#}$,
\item $\tfont{GLP}\vdash \wfont{B}^{\#}<_0\wfont{A}^{\#}$;
\end{enumerate}
\item from irreflexivity of $<_{Con}$ on $\omega$-correct theories (it follows from G\"odel Second Incompleteness Theorem) and transitivity of $<_{Con}$(it follows from arithmetical correctness of $\tfont{GLP}$ \cite[Lemma~5.3]{Bek05}) it follows that at most one of the following propositions holds:
\begin{enumerate}
\item $\wfont{A}^{\star}<_{Con}\wfont{B}^{\star}$,
\item $\wfont{A}^{\star}\equiv\wfont{B}^{\star}$,
\item $\wfont{B}^{\star}<_{Con}\wfont{A}^{\star}$.
\end{enumerate}
\end{enumerate}

From the partial arithmetic completeness of $\tfont{GLP}$ it follows that, for words $\wfont{A},\wfont{B},\wfont{C}\in \word_{\omega}$,  we have $$\wfont{A}^{\star}\equiv \wfont{B}^{\star}\; \Rightarrow\;(\wfont{A}\wfont{C})^{\star}\equiv (\wfont{B}\wfont{C})^{\star}.$$  Using our partial arithmetic completeness we reformulate some of results of \cite{Bek04-2}. From \cite[Lemma~1{\it(iv)}]{Bek04-2} it follows that, for a number $n\ge 0$, words $\wfont{A},\wfont{B}\in\sfont{S}_{n+1}$, and word $\wfont{C}\in\word_{\omega}$, the following holds:  $$\wfont{A}^{\star}\equiv \wfont{B}^{\star}\;\Rightarrow\;(\wfont{C}n\wfont{A})^{\star}\equiv (\wfont{C}n\wfont{B})^{\star}.$$ From \cite[Lemma~2]{Bek04-2} and \cite[Corollary 8]{Bek04-2} it follows that for numbers $n\ge 0$, $k\ge 2$ and  words $\wfont{A}_1,\wfont{A}_2,\ldots,\wfont{A}_{k}\in\sfont{S}_{n+1}$ such that $\wfont{A}_{k-1}^{\star}<_{Con}\wfont{A}_k^{\star}$ we have
$$(\wfont{A}_1n\ldots n\wfont{A}_{k-2}n\wfont{A}_{k-1}n\wfont{A}_k)^{\star}\equiv(\wfont{A}_1n\ldots n\wfont{A}_{k-2}n\wfont{A}_{k})^{\star}.$$
We consider the binary relation  $\ffont{R}$ on the set $\word_{\omega}$ $$\wfont{B}_1\; \ffont{R}\; \wfont{B}_2 \defiff (\wfont{B}_1^{\star}\equiv \wfont{B}_2^{\star}) \lor (\wfont{B}_1^{\star}<_{Con}\wfont{B}_2^{\star}).$$
From \cite[Proposition~3]{Bek04-2} and \cite[Proposition~4]{Bek04-2} we conclude that $\ffont{R}$ is a linear preorder on $\word_{\omega}$.
From this four facts we conclude that for a number $n\ge 0$ and words $\wfont{A}_1,\ldots,\wfont{A}_{k}\in\sfont{S}_{n+1}$ we have $$(\wfont{A}_1n\ldots n\wfont{A}_k)^{\star}\equiv(\wfont{C}_1n\ldots n\wfont{C}_f)^{\star},$$
where  $(\wfont{C}_1,\ldots,\wfont{C}_f)$ is the lexicographically maximal subsequence of the sequence $(\wfont{A}_1,\ldots ,\wfont{A}_k)$, with respect to  the linear preorder $\ffont{R}$.

We prove by induction on  $m-n$ that, for all  $m\ge n$ and $\wfont{A},\wfont{B}\in \sfont{S}_n\cap \word_m$, we have $$\wfont{A}\wore \wfont{B}\iff \wfont{A}\; \ffont{R}\; \wfont{B};$$  clearly,  from the induction hypothesis the fact follows. Obviously, the induction basis holds. Assume that the induction hypothesis holds for $n+1$ and $m$. We claim that for two $\ffont{R}$-monotone non-decreasing sequences $(\wfont{A}_1,\ldots,\wfont{A}_k)$ and $(\wfont{B}_1,\ldots,\wfont{B}_l)$ with all elements from $S_{n+1}\cap \word_m$ we have $$\begin{aligned}\mbox{ $(\wfont{B}_1,\ldots,\wfont{B}_l)$ is}&\mbox{ $\ffont{R}$-lexicographically not less than $(\wfont{A}_1,\ldots,\wfont{A}_k)$} \;\Rightarrow\;\\ &\ \wfont{A}_1n\ldots n\wfont{A}_k\; \ffont{R}\; \wfont{B}_1n\ldots n\wfont{B}_l.\end{aligned}$$ We consider two sequences  $(\wfont{A}_1,\ldots,\wfont{A}_k)$ and $(\wfont{B}_1,\ldots,\wfont{B}_l)$ as above such that the sequence $(\wfont{B}_1,\ldots,\wfont{B}_l)$ is $\ffont{R}$-lexicographically not less than $(\wfont{A}_1,\ldots,\wfont{A}_k)$  and show that $$\wfont{A}_1n\ldots n\wfont{A}_k\; \ffont{R}\; \wfont{B}_1n\ldots n\wfont{B}_l.$$ Clearly, for some $s$ from $0$ to $\min(r,l)$, the sequence $(\wfont{A}_1,\ldots,\wfont{A}_s,\wfont{B}_{s+1},\ldots,\wfont{B}_l)$ is the lexicographically maximal subsequence of $(\wfont{A}_1,\ldots,\wfont{A}_s,\wfont{A}_{s+1},\ldots,\wfont{A}_{k},\wfont{B}_{s+1},\ldots,\wfont{B}_l)$  and the words $\wfont{A}_1,\ldots,\wfont{A}_s$ are $\ffont{R}$-equivalent to the words $\wfont{B}_1,\ldots,\wfont{B}_s$, respectively. Obviously, for every $\wfont{C},\wfont{D}\in\word_{\omega}$, we have $\wfont{C} \; \ffont{R}\;\wfont{C}\wfont{D}$.  Thus, $$\wfont{A}_1n\ldots n\wfont{A}_k\; \ffont{R}\; \wfont{A}_1n\ldots n\wfont{A}_sn\wfont{B}_{s+1}n\ldots n\wfont{B}_l.$$ Therefore, because $ (\wfont{A}_1n\ldots n\wfont{A}_sn\wfont{B}_{s+1}n\ldots n\wfont{B}_l)^{\star}\equiv ( \wfont{B}_1n\ldots n\wfont{B}_l)^{\star}$, we have the required  $$\wfont{A}_1n\ldots n\wfont{A}_k\; \ffont{R}\; \wfont{B}_1n\ldots n\wfont{B}_l.$$ Because $\ffont{R}$ is a linear preorder, the induction hypothesis for $n$ and $m$ follows from the claim.\end{proof}

We define operators $a_0$, $a_1$,$\ldots$ on $\word_{\omega}$: $$a_n\colon \wfont{\wfont{A}}\longmapsto \wfont{A} n.$$ From Fact \ref{word_theory_correspondence} it follows that the structures $(\mathfrak{S}_\omega,\tfont{T}_0,<_{Con},\refl{0},\refl{1},\ldots)$ and $(\word_{\omega},\varLambda,\wor,a_0,a_1,\ldots)$ are isomorphic. 

\subsection{Properties of words comparison} In this subsection we prove some basic properties of $\wor$. Some of them were known before and were proved using the definition based on the Japaridze's provability logic. We prove these properties using our combinatorial definition.

In the proofs in the present subsection we need several technical notions. We consider monotonically increasing finite sequences of non-zero natural numbers; we call them {\it index collections}. We say that an index collection $(s_1,\ldots,s_m)$ is {\it $n$-bounded}, if every  $s_i\le n$. Every $n$-bounded index collection $(s_1,\ldots,s_m)$ corresponds to the subsequence $(\wfont{A}_{s_1},\ldots,\wfont{A}_{s_m})$ of a sequence $(\wfont{A}_1,\ldots,\wfont{A}_n)$; note that a subsequence of a sequence can corresponds to more than one index collection. We say that $n$-bounded index collection  is {\it maximal for $(\wfont{A}_1,\ldots,\wfont{A}_n)$}, if it corresponds to the lexicographically maximal subsequence of $(\wfont{A}_1,\ldots,\wfont{A}_n)$.

\begin{lemma} \label{maximal_lexicographic_algorithm} Suppose $(\wfont{A}_1,\ldots,\wfont{A}_n)$ is a sequence of words from $\word_{\omega}$. Then there exists the unique $n$-bounded index collection $(s_1,\ldots,s_m)$ that is maximal for $(\wfont{A}_1,\ldots,\wfont{A}_n)$.  $m,s_1,s_2,\ldots,s_m$ are determined by the following equations for $m,s_0,s_1,\ldots$:
\begin{enumerate}
\item \label{maximal_lexicographic_algorithm_eq1}  $s_0=0$;
\item \label{maximal_lexicographic_algorithm_eq2}  $s_{k+1}=\min\{f\in \omega \mid \forall l\in \omega (s_k<l\le n \to \wfont{A}_l\wore \wfont{A}_{f})\}$, if $s_k\ne n$;
\item \label{maximal_lexicographic_algorithm_eq3} $s_{k+1}=n$, if $s_k= n$;
\item \label{maximal_lexicographic_algorithm_eq4}  $m=\min\{f\ge 1 \mid s_f=n\}$.
\end{enumerate}
\end{lemma}
\sprv Suppose the numbers $m,s_0,s_1,\ldots$ are given by the equations \ref{maximal_lexicographic_algorithm_eq1}, \ref{maximal_lexicographic_algorithm_eq2}, \ref{maximal_lexicographic_algorithm_eq3}, and \ref{maximal_lexicographic_algorithm_eq4}. Note that $(s_1,\ldots,s_m)$  is an $n$-bounded index collection.

By induction on $k$ we show that the only possible first $\min(k,m)$ indexes of an $n$-bounded index collection that is maximal for $(\wfont{A}_1,\ldots,\wfont{A}_n)$ are $s_1,\ldots,s_{\min(k,m)}$. The induction basis ($k=0$) and the induction step  in the case of  $k>m$ obviously holds. Suppose the induction hypothesis holds for $k-1$. We claim that, for an index $h$ from $s_{k-1}+1$ to $n$ such that $h\ne s_k$, the  index collection $(s_1,\ldots,s_{k-1},h)$ is not a prefix of some $n$-bounded index collection that is maximal for $(\wfont{A}_1,\ldots,\wfont{A}_n)$; clearly, the induction hypothesis for $k$ follows from the claim. In the case of $\wfont{A}_h\nsim \wfont{A}_{s_k}$, we have $\wfont{A}_h\wor \wfont{A}_{s_k}$, hence the sequence $(\wfont{A}_{s_1},\ldots,\wfont{A}_{s_k})$ is lexicographically greater than any sequence with a prefix that is equal to $(\wfont{A}_{s_1},\ldots,\wfont{A}_{s_{k-1}},\wfont{A}_h)$; therefore, in this case, the claim holds. Let us consider the case of $\wfont{A}_h\sim \wfont{A}_{s_k}$. Obviously, $h> s_k$. Hence, for every $n$-bounded index collection $(s_1,\ldots,s_{k-1},h,u_1,\ldots ,u_l)$, the corresponding subsequence is lexicographically less than the subseqence that corresponds to the index collection $(s_1,\ldots,s_k,h,u_1,\ldots ,u_l)$. Thus $(s_1,\ldots,s_{k-1},h)$ is not a prefix of an $n$-bounded index collection that is maximal for $(\wfont{A}_1,\ldots,\wfont{A}_n)$.\eprv

\begin{lemma} \label{concatenation_lemma} Suppose $(\wfont{A}_1,\ldots,\wfont{A}_n)$ and $(\wfont{B}_1,\ldots,\wfont{B}_m)$ are  non-empty word sequences and maximal index collections for them are an  $n$-bounded index collection $(g_1,\ldots,g_r)$ and an $m$-bounded index collection $(h_1,\ldots,h_t)$, respectively. Then the $(n+m)$-bounded index collection $(g_1,\ldots,g_k,n+h_1,\ldots,n+h_t)$ is maximal for the sequence  $(\wfont{A}_1,\ldots,\wfont{A}_n,\wfont{B}_1,\ldots,\wfont{B}_m)$, where $k=\max(\{0\}\cup\{i\mid 1\le i\le r, \wfont{B}_{h_1}\wore \wfont{A}_{g_i}\})$.
\end{lemma}
\sprv We put $\wfont{C}_1=\wfont{A}_1,\ldots,\wfont{C}_n=\wfont{A}_n,\wfont{C}_{n+1}=\wfont{B}_1,\ldots,\wfont{C}_{n+m}=\wfont{B}_m$. We consider the sequence $s_i$ that is given by equations from Lemma \ref{maximal_lexicographic_algorithm} for the sequence $(\wfont{C}_1,\ldots,\wfont{C}_{n+m})$. Let us prove that $s_1=g_1,\ldots,s_k=g_k,s_{k+1}=n+h_1,\ldots,s_{k+t}=n+h_{k+t}$; clearly, the later is equivalent to the lemma.

 We put $g_0=0$. Hence, for $i$ from $1$ to $k$, we have $g_i=\min\{f\in \omega \mid \forall l\in \omega (g_{i-1}<l\le n \to \wfont{A}_l\wore \wfont{A}_{f}\}$. From Lemma \ref{maximal_lexicographic_algorithm} it follows that $\wfont{B}_{h_1}$ is the maximal element of the sequence $(\wfont{B}_1,\ldots,\wfont{B}_m)$. Thus, for $i$ from  $1$ to $k$, we have $\forall l \in \{n+1,\ldots,n+m\}(\wfont{C}_l\wore \wfont{C}_{g_i})$. Therefore, for $i$ from $1$ to $k$, we have $s_i=\min\{f\in \omega \mid \forall l\in \omega (g_{i-1}<l\le n \to \wfont{C}_l\wore \wfont{C}_{f})\}=g_i$. 

Note that if $k=r$, then $s_k=g_k=n$, hence the required straightforward follows from Lemma \ref{maximal_lexicographic_algorithm}. Now we consider the case of $k<r$. For  $i$ from   $s_k+1$ to $n$,  we have $\wfont{C}_i\wore \wfont{C}_{s_{k+1}}$. From the definition of $k$ it follows that  $\wfont{C}_{s_{k+1}}\wor \wfont{C}_{n+h_1}$. Thus, for every $i$ from $s_{k+1}+1$ to $n$, by transitivity of $\wore$, we have $\wfont{C}_i\wore \wfont{C}_{n+h_1}$. Therefore $s_{k+1}=n+h_1$. From Lemma \ref{maximal_lexicographic_algorithm} it follows that, for all $i$ from $1$ to $t$, we have $s_{k+i}=n+h_i$. It completes the proof of the lemma.\eprv

\begin{lemma} \label{congruency_lemma} Suppose $k\in \omega$, $\wfont{A},\wfont{B}\in \word_{\omega}$ and $\wfont{A}\sim \wfont{B}$. Then $\wfont{A} k\sim \wfont{B} k$.
\end{lemma} 
\sprv We prove the lemma for all $\wfont{A},\wfont{B}\in \word_{\omega}$ by induction on the length of $\wfont{A}\wfont{B}$. Induction basis  obviously holds, i.e. the case of $\wfont{A}\wfont{B}=\varLambda$. Let us prove the induction step. Suppose the minimal symbol of $\wfont{A}\wfont{B}$ is $n$. Thus if $k<n$, then the comparison of  $\wfont{A} k$ and  $\wfont{B} k$ can be reduced to the lexicographical compare of sequences $(\wfont{A},\varLambda)$ and $(\wfont{B},\varLambda)$; the last two sequences, obviously, are lexicographically equivalent. Further we assume that $k\ge n$. 

We consider the only $q,l\in \omega$, $\wfont{A}_1,\ldots,\wfont{A}_q\in \sfont{S}_{n+1}$, and $\wfont{B}_1,\ldots,\wfont{B}_l\in \sfont{S}_{n+1}$ such that $\wfont{A}=\wfont{A}_1n\wfont{A}_2n\ldots n\wfont{A}_q$ and $\wfont{B}=\wfont{B}_1n\wfont{B}_2n\ldots n\wfont{B}_l$. Suppose $(s_1,\ldots,s_r)$ and $(h_1,\ldots,h_t)$ are maximal index collections for $(\wfont{A}_1,\ldots,\wfont{A}_l)$ and $(\wfont{B}_1,\ldots,\wfont{B}_q)$, respectively. Note that from $\wfont{A}\sim \wfont{B}$ it follows that we have $r=t$ and $\wfont{A}_{s_i}\sim \wfont{B}_{h_i}$ for all $i$ from $1$ to $r$ 

We consider the case of $k=n$. In order to compare $\wfont{A} k$ and $\wfont{B} k$ we need to compare lexicographically maximal subsequences  of sequences  $(\wfont{A}_1,\ldots,\wfont{A}_f,\varLambda)$ and $(\wfont{B}_1,\ldots,\wfont{B}_g,\varLambda)$. From Lemma \ref{concatenation_lemma} it follows that this sequences are equal to $(\wfont{A}_{s_1},\ldots,\wfont{A}_{s_r},\varLambda)$ and $(\wfont{B}_1,\ldots,\wfont{B}_{h_t},\varLambda)$, respectively. Thus from equivalency of the words $\wfont{A}$ and $\wfont{B}$ it follows that $\wfont{A} k$ and $\wfont{B} k$ are equivalent.

 Now we consider the case of $k>n$. Note that $s_r=f$ and $h_t=g$. Hence $\wfont{A}_f\sim \wfont{B}_g$. Therefore from the induction hypothesis it follows that $\wfont{A}_f k\sim \wfont{B}_g k$. We consider the index collections $(s^{'}_1,\ldots,s^{'}_{r^{'}})$ and $(h^{'}_1,\ldots,h^{'}_{t^{'}})$ that are maximal for sequences $(\wfont{A}_1,\ldots,\wfont{A}_{f-1},\wfont{A}_f k)$ and $(\wfont{B}_1,\ldots,\wfont{B}_{g-1},\wfont{B}_g k)$, respectively. From Lemma \ref{maximal_lexicographic_algorithm} it follows that $r^{'}=\max(\{0\}\cup \{i\in \{1,\ldots,r\} \mid \wfont{A}_f k \wore \wfont{A}_{s_i}\})+1$, $s_i=s^{'}_i$, for $i$ from $1$ to $r^{'}-1$, and $s_{r^{'}}=f$. Similarly, $t^{'}=\max(\{0\}\cup \{i\in \{1,\ldots,t\} \mid \wfont{B}_g k \wore \wfont{B}_{h_i}\})+1$, $h_i=h^{'}_i$, for $i$ from $1$ to $t^{'}-1$, and $h_{t^{'}}=g$. Therefore $t^{'}=r^{'}$, $$\wfont{A}_{s_i}= \wfont{A}_{s_i'}\sim\wfont{B}_{h_i'}=\wfont{B}_{h_i}\mbox{, for $i\in \{1,\ldots,r^{'}-1\}$},$$ and $\wfont{A}_{s^{'}_{r^{'}}}=\wfont{B}_{h^{'}_{r^{'}}}$. Thus $\wfont{A} k\sim \wfont{B} k$.\eprv

We define the set of all words in normal form $\NF$. We define the property ``$\wfont{A}$ is an element of $\NF$'' by induction on the length of $\wfont{A}$:
\begin{itemize}
\item $\varLambda\in \NF$;
\item suppose $n$ is a number, $k\ge 2$, and $\wfont{A}_1,\ldots,\wfont{A}_k\in \sfont{S}_{n+1}$, then $\wfont{A}_1n\ldots n \wfont{A}_{k}\in \NF$ iff $\wfont{A}_k  \wore \wfont{A}_{k-1} \wore \ldots \wore \wfont{A}_1$ and $\wfont{A}_1,\ldots,\wfont{A}_k\in \NF$.
\end{itemize}

By trivial induction on the length of a word $\wfont{A}$, we prove that there exists the unique  $\wfont{B}\in\NF$ such that $|\wfont{B}|\le|\wfont{A}|$ and $\wfont{A}\sim\wfont{B}$. Therefore, for every $\wfont{A}\in \word_{\omega}$, there exists the unique $\wfont{B}\in \NF$ such that $\wfont{A}\sim\wfont{B}$.

We introduce operators $\romb{n}$ on the set  $\NF$. For every $\wfont{A}\in \NF$ we consider $\wfont{B}$ such that $\wfont{B}\sim \wfont{A} n$ and $\wfont{B}\in \NF$; we put $\romb{n}\wfont{A}=\wfont{B}$.

Note that the restriction of $\wore$ to $\NF$ is a  non-strict linear order, and the restriction of  $\wor$ to  $\NF$ is a strict linear order. 

For every $\alpha \in [0,\omega]$, we denote by $\word^N_{\alpha}$ the set $\word_{\alpha}\cap\NF$.

From Lemma \ref{congruency_lemma} it follows that

\begin{proposition}\label{a_diamond_correspondence} For every $n\in \omega$ and $\wfont{A}_1,\wfont{A}_2\in \word_\omega$ such that $\wfont{A}_1\sim \wfont{A}_2$, we have $a_n(\wfont{A}_1)\sim a_n(\wfont{A}_2)$. Moreover, for all  $n\in\omega$, $\wfont{A}_1\in \word_\omega^N$ and $\wfont{A}_2\in \word_\omega$ such that $\wfont{A}_1\sim\wfont{A}_2$, we have $\romb{n}\wfont{A}_1 \sim a_n(\wfont{A}_2)$
\end{proposition}

From Proposition \ref{a_diamond_correspondence} it follows that the structures $(\word_{\alpha}/{\sim},\wor,\langle a_i\mid i\in\omega, i\le \alpha\rangle)$ and  $(\word_{\alpha}^N,\wor,\langle \romb{i}\mid i\in\omega, i\le \alpha\rangle)$ are isomorphic, for all $\alpha\in[0,\omega]$.

\begin{lemma}\label{comparission_lemma} Suppose $\wfont{A},\wfont{B},\wfont{C},\wfont{D} \in \word_{\omega}$ are such that the length of $\wfont{A}$ is equal to the length of $\wfont{B}$ and, for all symbols $c_1$ and $c_2$ that lies in positions with the same indexes in $\wfont{A}$ and $\wfont{B}$, respectively, we have $c_1\le c_2$. Then $\wfont{A}\wore \wfont{D}\wfont{B}\wfont{C}$. Moreover, if either $\wfont{C}\ne \varLambda$ or the last symbols of $\wfont{A}$ and $\wfont{B}$ are different, then $\wfont{A}\wor \wfont{D}\wfont{B}\wfont{C}$.
\end{lemma}
\sprv We prove the lemma by induction on the length of $\wfont{D}\wfont{B}\wfont{C}$. The induction basis, i.e. the case of $|\wfont{D}\wfont{B}\wfont{C}|=0$, is trivial. Now we prove the induction step.

Suppose $n$ is the minimal symbol of $\wfont{D}\wfont{A}\wfont{C}$  and the word $\wfont{A}$ have the form $\wfont{A}_1n\wfont{A}_{2}n\ldots n \wfont{A}_k$, where $\wfont{A}_1,\ldots,\wfont{A}_k\in\sfont{S}_{n+1}$. Suppose $(s_1,\ldots,s_l)$ is the index collection that is  maximal for $(\wfont{A}_1,\ldots,\wfont{A}_k)$. For every  $i$ from $1$ to $l$, the word $\wfont{B}$ have the form $\wfont{E}_i\wfont{F}_i\wfont{G}_i$, for some words  $\wfont{E}_i$, $\wfont{F}_i$, $\wfont{G}_i$ such that $|\wfont{E}_i|=|\wfont{A}_1n\ldots n\wfont{A}_{s_i-1}n|$, $|\wfont{F}_i|=|\wfont{A}_{s_i}|$, $|\wfont{G}_i|=|n\wfont{A}_{s_i+1}n\ldots n \wfont{A}_{s_k}|$. We consider the minimal $u\in\{1,\ldots,l\}$ such that either the first symbol of $\wfont{G}_u$ is not equal to $n$ or $u$ is equal to $l$. For $i$ from $1$ to $u$, we denote by  $\wfont{H}_i$ the longest postfix of $\wfont{E}_i$ without symbol $n$. 

We choose $\wfont{B}_1,\ldots,\wfont{B}_f\in \sfont{S}_{n+1}$ such that $\wfont{B}_1n\wfont{B}_{2}n\ldots n\wfont{B}_f$ is equal to $\wfont{D}\wfont{B}\wfont{C}$. We find the minimal $ g_1,g_2,\ldots,g_u$ such that, for all $i$ from $1$ to $u$, we have $$\sum\limits_{j=1,2,\ldots ,g_i} |\wfont{B}_jn|\ge|\wfont{D}|+|\wfont{E}_i|+|\wfont{F}_i|+1.$$  Clearly, for all $i$ from $1$ to $u-1$, we have $$\sum\limits_{j=1,2,\ldots ,g_i} |\wfont{B}_jn|=|\wfont{D}|+|\wfont{E}_i|+|\wfont{F}_i|+1$$  and $1\le g_1<g_2<\ldots<g_u\le f$. Note that $\wfont{B}_{g_i}=\wfont{H}_i\wfont{F}_i$, for $i$ from $1$ to $u-1$, and $\wfont{B}_u$ is equal to $\wfont{H}_i \wfont{F}_i \wfont{I}$, for some  $\wfont{I}\in \sfont{S}_{n+1}$. 

Clearly, if the sequence $(\wfont{B}_{g_1},\ldots,\wfont{B}_{g_u})$ is lexicographically greater (not less) than the sequence$(\wfont{A}_{s_1},\ldots,\wfont{A}_{s_l})$, then $\wfont{A}\wor \wfont{D}\wfont{B}\wfont{C}$($\wfont{A}\wore \wfont{D} \wfont{B}\wfont{C}$). By induction hypothesis, we have $\wfont{A}_{s_i}\wore \wfont{H}_i\wfont{F}_i$, for $i$ from $1$ to $u$. Thus the sequence $\wfont{B}_{g_1},\ldots,\wfont{B}_{g_u}$ is lexicographically not less than $\wfont{A}_{s_1},\ldots,\wfont{A}_{s_u}$. If, moreover, $\wfont{I}\ne\varLambda$ or the last symbol of $\wfont{F}_u$ is not equal to the last symbol of $\wfont{A}_{s_u}$ then by induction hypothesis $\wfont{A}_{s_u}\wor \wfont{H}_u\wfont{F}_u \wfont{I}$, hence $\wfont{A}_{s_u}\wor \wfont{B}_{g_u}$, and therefore $\wfont{A}\wor \wfont{D}\wfont{B}\wfont{C}$.

If  $u<l$, then $\wfont{I}\ne \varLambda$, and hence $\wfont{A}\wor \wfont{D} \wfont{B}\wfont{C}$. 

Let us consider the case of $u=l$. If $\wfont{C}=\varLambda$ and the last symbols of $\wfont{B}$ and $\wfont{A}$ are equal, then we already have $\wfont{A}\wore \wfont{D} \wfont{B}\wfont{C}$. If either $\wfont{C}\ne\varLambda$ or $\wfont{G}_u\ne \varLambda$ then either  $\wfont{I}\ne\varLambda$ or $f>g_u$. If $\wfont{I}\ne\varLambda$, then $\wfont{A}\wor \wfont{D} \wfont{B}\wfont{C}$. If $f>g_u$, then, because $(\wfont{B}_{g_1},\ldots,\wfont{B}_{g_u},\wfont{B}_f)$ is lexicographically greater than  $(\wfont{A}_{s_1},\ldots,\wfont{A}_{s_l})$,  we have $\wfont{A}\wor \wfont{D}\wfont{B}\wfont{C}$. Now we consider the last case: $\wfont{C}=\varLambda$, $\wfont{G}_u= \varLambda$, and the last symbols of $\wfont{A}$ and $\wfont{B}$ are not equal. Note that in this case the last symbols of $\wfont{F}_u$ and $\wfont{A}_{s_u}$ are not equal too. Hence, by induction hypothesis, $\wfont{A}_{s_u}\wor\wfont{H}_u\wfont{F}_u=\wfont{B}_{g_u}$. Thus $\wfont{A}\wor \wfont{D}\wfont{B}\wfont{C}$. This finishes the proof of the lemma.\eprv


\section{Ordinal notation systems with undecidable elementary theories}
In this section we prove that for all  $\alpha$ from $3$ to $\omega$ the theory $\Th(\word_\alpha^N,\wor,\romb{1},\romb{3})$ is decidable. We will prove that for all $\alpha$ from $3$ to $\omega$ the set $\word_3^N$ is first-order definable in $(\word_\alpha^N,\wor,\romb{1},\romb{3})$. After that, we use the technique based on hereditary undecidable theories and right total interpretations in order to prove $(\word_\alpha^N,\wor,\romb{1},\romb{3})$. The elementary theory $\Th(\sfont{L}^2_{\textit{fin}})$ of all finite sets with a pair of linear orders on them is hereditary undecidable \cite{Lav63}. We show that there exists a relative right total interpretation of $\Th(\sfont{L}^2_{\textit{fin}})$ in $\Th(\word_3^N,\wor,\romb{1},\romb{3})$. From the late straightforward follows the undecidability of $\Th(\word_3^N,\wor,\romb{1},\romb{3})$. Thus for every  $\alpha\in [3,\omega]$ the elementary theory $\Th(\word_\alpha^N,\wor,\romb{1},\romb{3})$ is undecidable. 

In this section and further we consider theories in model theoretic manner, i.e. as sets of propositions of a signature $\sigma$ (signature of a theory) that include all theorems of predicate calculus for signature $\sigma$ and is closed under the rule {\it Modus Ponens}. Here we use predicate calculus with equality and don't include equality  symbol in signatures.

The {\it elementary theory} of a class of structures $\mathbf{A}$ with the same signature $\sigma$ is the set of all propositions of signature $\sigma$ that are true in all models of the class $\mathbf{A}$. We denote the elementary theory of a class of models  $\mathbf{A}$ with the same signature by $\Th(\mathbf{A})$. The {\it elementary theory} of a model $\mathfrak{A}$ is  $\Th(\{\mathfrak{A}\})$; we denote it by $\Th(\mathfrak{A})$. 

Suppose we have a model $\mathfrak{A}$ with domain $\sfont{A}$. A set $\sfont{E}\subset \underbrace{\sfont{A}\times \sfont{A} \times \ldots \times \sfont{A}}\limits_{\mbox{$n$ times}}$ is {\it definable} in $\mathfrak{A}$ if there is a first-order formula $\ffont{F}(x_1,\ldots,x_n)$ of the signature of the model $\mathfrak{A}$ such that, for all $a_1,\ldots,a_n\in \sfont{A}$ we have
$$(a_1,\ldots,a_n)\in \sfont{E} \iff \mathfrak{A} \models \ffont{F}[a_1,\ldots,a_n/x_1,\ldots,x_n].$$   
We consider every $n$-ary predicate  $\sfont{A}$ as a subset of $\underbrace{\sfont{A}\times \sfont{A} \times \ldots \times \sfont{A}}\limits_{\mbox{$n$ times}}$. Also we consider every function $$f\colon \sfont{D}\to \sfont{A}\mbox{, where }\sfont{D}\subset \underbrace{\sfont{A}\times \sfont{A} \times \ldots \times \sfont{A}}\limits_{\mbox{$n$ times}}$$ as the subset $$\{(a_1,\ldots,a_n,f(a_1,\ldots,a_n))\mid (a_1,\ldots,a_n)\in \sfont{D}\}\subset\underbrace{\sfont{A}\times \sfont{A} \times \ldots \times \sfont{A}}\limits_{\mbox{$n+1$ times}}.$$ Thus we can talk about first-order definability of predicates and function in $\mathfrak{A}$.

\begin{lemma} \label{interpretability_lemma} For an ordinal $\alpha$ from $3$ to $\omega$ the set $\word_3^N$ is definable in $(\word_\alpha^N,\wor,\romb{1},\romb{3})$.
\end{lemma}
\sprv For  $\alpha=3$ the lemma obviously holds. Let us prove the lemma in the case of $\alpha\ge 4$.  We consider property of an element $x\in\word_\alpha^N$: $$x\ne\varLambda\&\forall y \wor x (\romb{3} y \wor x).$$ Let us prove that the one symbol word $4$ is the first element $x\in\word_\alpha^N$ such that it satisfies the property under consideration. We claim that for any word $\wfont{A}\in \word_\alpha^N$ we have 
 $$\wfont{A}\in \word_3^N\iff \wfont{A}\wor 4.$$ From Lemma \ref{comparission_lemma} it follows that the word $4$ is the minimal element of $\word_\alpha^N\setminus \word_3^N$.  Note that from Lemma \ref{comparission_lemma}  it follows that, for every word $\wfont{B}\in \word_3$, there is $n$ such that $\wfont{B}\wor 3^n$. Obviously, for every number $n$, we have $3^n\wor 4$. Therefore, for every $\wfont{B}\in \word_3^N$, we have $\wfont{B}\wor 4$, $\romb{3}\wfont{B}\sim \wfont{B} 3\wor 4$, hence the claim holds.  Hence the word $4$ satisfies the required property. Every $\wfont{B}\in \word_3^N\setminus \{\varLambda\}$ is equal to $\wfont{C} k$, for some $k\le 3$ and $\wfont{C}\in\word_3^N$, hence $\romb{3}\wfont{C} \sim \wfont{C} 3 \wore \wfont{B}$. Also from Lemma \ref{comparission_lemma} it follows that $\wfont{C}\wor \wfont{B}$. Thus every element of $\word_3^N$ doesn't satisfies the property under consideration. Hence the word $4$ is the minimal element of the set $\word_\alpha^N$ that satisfies the property under consideration.

Thus in $(\word_\alpha^N,\wor,\romb{1},\romb{3})$ the element $4$ is definable.  Above we have showed that for any word $\wfont{A}\in \word_\alpha^N$ we have 
 $$\wfont{A}\in \word_3^N\iff \wfont{A}\wor 4.$$
The late gives us the required definition.\eprv

Suppose  $\wfont{A}\in \word_3^N\setminus \{\varLambda\}$ and $\wfont{B}\in \word_3^N\setminus \{\varLambda\}$ are words such that $\wfont{A}=\wfont{C}_10\wfont{C}_{2}0\ldots 0\wfont{C}_n$, $\wfont{B}=\wfont{C}_1 0 \wfont{C}_{2}0 \ldots 0 \wfont{C}_m$, for some $n\ge m \ge 1$,  $\wfont{C}_i\in \sfont{S}_{1}$. In this case we say that $\wfont{B}$ is a slice of $\wfont{A}$. We give the predicate $\ffont{Sl}(x,y)$ as  the following:
$$\ffont{Sl}(\wfont{A},\wfont{B})\defiff \mbox{$\wfont{B}$ is a slice of $\wfont{A}$.}$$

\begin{lemma} \label{Sl_lemma} The predicate $\ffont{Sl}(x,y)$ is definable in the model $(\word_3^N,\wor,\romb{1},\romb{3})$.
\end{lemma}
\sprv  Let us prove that for all $\wfont{A},\wfont{B}\in \word_3^N$:
\begin{equation}\label{Sl_lemma_eq}\ffont{Sl}(\wfont{A},\wfont{B}) \iff \wfont{A}\ne \varLambda\&  \wfont{B} \ne \varLambda \& \wfont{B} \wore \wfont{A} \& \wfont{A} \wor \romb{1} \wfont{B}.\end{equation}

We consider words $\wfont{A},\wfont{B}\in \word_3^N$. Suppose  we have numbers $n$, $m$  and words $\wfont{A}_1,\ldots,\wfont{A}_n, \wfont{B}_1,\ldots,\wfont{B}_m\in \sfont{S}_1\cap \NF$ such that $\wfont{A}=\wfont{A}_10\wfont{A}_{2}0\ldots 0 \wfont{A}_n$ and $\wfont{B}=\wfont{B}_10\wfont{B}_20\ldots 0\wfont{B}_m$. 

Assume that $\wfont{B}$ is a slice of $\wfont{A}$. Let us prove that the right side of (\ref{Sl_lemma_eq}) holds. From our assumption we conclude that $\wfont{A}\ne \varLambda$ and $\wfont{B}\ne \varLambda$. Also, from the assumption it follows that $m\le n$ and $\wfont{B}_i=\wfont{A}_i$, for all $i$ from $1$ to $m$. From the definition of $\NF$ it follows that $\wfont{A}_j\wore \wfont{A}_i$, for  $1\le i \le j \le n$. Therefore, because of Remark \ref{wore_characterization}, we have $\wfont{B}\wore \wfont{A}$. Note that $\wfont{A}_m\wor \wfont{A}_m 1$  and the lexicographically maximal subsequence of the sequence $(\wfont{A}_1,\wfont{A}_2,\ldots \wfont{A}_{m-1},\wfont{A}_m 1)$ is equal to $(\wfont{A}_1,\wfont{A}_{2},\ldots,\wfont{A}_s,\wfont{A}_m1)$, for some $0\le s < m$ and $\wfont{A}_{s+1}\wor \wfont{A}_m 1$. Therefore $\wfont{A} \wor \wfont{B} 1$, hence $\wfont{A}\wor \romb{1} \wfont{B}$.

Now we assume that a pair $(\wfont{A}$, $\wfont{B})$ satisfies the right side of  (\ref{Sl_lemma_eq}). Let us prove that $\wfont{B}$ is a slice of $\wfont{A}$. From the conditions  $\wfont{B} \ne \varLambda$ and $\wfont{B} \wore \wfont{A}$ it follows that there exists a natural number $l$ from $0$ to $\min(m,n)$ such that, for all $i$ from $1$ to $l$, we have $\wfont{A}_i=\wfont{B}_i$. Also, either $l=m\le n$ or both $l<\min(m,n)$ and $\wfont{B}_{l+1}\wor \wfont{A}_{l+1}$. Let us prove by contradiction that $l=m\le n$. Assume that $l<\min(m,n)$ and $\wfont{B}_{l+1}\wor \wfont{A}_{l+1}$. Clearly, we have $\romb{1}\wfont{C}\wore \wfont{I}$ and $\romb{1}\wfont{C}=\wfont{C} 1$, for all $\wfont{C},\wfont{I}\in \sfont{S}_1\cap\NF$ such that $\wfont{C}\wor \wfont{I}$. Let us prove that $\wfont{B} 1\wore \wfont{A}_1 0\ldots 0 \wfont{A}_l 0\wfont{B}_{l+1}1$. We consider the lexicographically maximal subsequence of the sequence $(\wfont{B}_{1},\ldots,\wfont{B}_{m-1},\wfont{B}_m 1)$. In the case of $\wfont{B}_m=\wfont{B}_{l+1}$ this subsequence can be given in the form $(\wfont{A}_1,\ldots,\wfont{A}_l,\wfont{B}_{l+1} 1)$. In the case of $\wfont{B}_m\wor\wfont{B}_{l+1}$ this subsequence can be given in the form $(\wfont{A}_1,\ldots,\wfont{A}_l,\wfont{B}_{l+1},\ldots,\wfont{B}_s,\wfont{B}_m1)$, for some $s$ from $l$ to $m-1$. Obviously, in both cases $\wfont{B} 1\wore \wfont{A}_1 0 \ldots 0 \wfont{A}_l 0 \wfont{B}_{l+1}1$.  Hence $\romb{1} \wfont{B} \sim  \wfont{B} 1\wore \wfont{A}_1 0 \ldots 0 \wfont{A}_l 0\wfont{B}_{l+1}1\wore \wfont{A}_1 0 \ldots 0 \wfont{A}_l 0 \wfont{A}_{l+1} \wore \wfont{A}$. The late contradicts with $\wfont{A}\wor\romb{1}\wfont{B}$.  Hence $l=m\le n$. Therefore the left side of (\ref{Sl_lemma_eq}) holds.\eprv

 We denote by $\wfont{I}_s$ the word $ 3^s 2$. For natural numbers  $k$ and $h$ such that $1\le k \le h$, we denote by $\wfont{K}_{h,k}$ the word $\wfont{I}_{h-1}\ldots\wfont{I}_{k+1}\wfont{I}_k3^k$, and we denote  by $\wfont{L}_h$ the word $\wfont{I}_{h-1}\ldots\wfont{I}_1$. Note that, for all  $k$ and $h$ such that $1\le k\le h$, we have 
$$\wfont{I}_{h-1}\ldots \wfont{I}_k 3^{k-1}\wore \wfont{L}_h \wor \wfont{I}_{h-1}\ldots \wfont{I}_{k}3^k=\wfont{K}_{h,k}.$$
Note that $$\wfont{K}_{h,1}\wor \wfont{K}_{h,2}\wor \ldots \wor \wfont{K}_{h,h}.$$
Suppose $\wfont{A}\in \word_3^N$ is equal to $\wfont{A}_r0\wfont{A}_{r-1}0\ldots 0\wfont{A}_1$, where all $\wfont{A}_i\in \sfont{S}_1\cap \word_3^N$. We put in the correspondence with $\wfont{A}$ the finite sequence of words $\textbf{u}(\wfont{A})=(u_1(\wfont{A}),\ldots,u_r(\wfont{A}))$, where for every $i\in\{1,\ldots,r\}$ we put $u_i(\wfont{A})=\romb{3}\wfont{A}_{r}0\wfont{A}_{r-1}0\ldots 0\wfont{A}_i$.

\begin{lemma} \label{seq_existence} Suppose we have non-zero natural numbers $h\ge 1$, $r$ and a collection of natural numbers $k_1,\ldots,k_r\le h$. Then there exists a word $\wfont{A}\in \word_3^N$ such that the sequence $\textbf{u}(\wfont{A})$ is equal to the sequence $\wfont{K}_{h,k_1},\ldots,\wfont{K}_{h,k_r}$.
\end{lemma}

\sprv For $i$ from $1$ to $r$ we denote by $\wfont{C}_i$ the word $\wfont{I}_{h-1}\ldots \wfont{I}_{k_i}3^{k_i-1}$. We put: 
 $$\wfont{A}=(\wfont{L}_h 1)^{r-1}\wfont{C}_r0(\wfont{L}_h 1)^{r-2}\wfont{C}_{r-1}\ldots (\wfont{L}_h 1)^0\wfont{C}_1.$$
Let us consider a number  $i$ from $1$ to $r$. The word $u_i(\wfont{A})$ is equal to the normal form of the word $(\wfont{L}_h1)^{r-1}\wfont{C}_r0(\wfont{L}_h1)^{r-2}\wfont{C}_{r-1}\ldots (\wfont{L}_h1)^{i-1}\wfont{C}_i3$. Because $\wfont{L}_h \wor \wfont{K}_{h,k_i}$,  we, using the definition of $\wore$, conclude that $(\wfont{L}_h 1)^{i-1} \wfont{C}_{i}3\sim \wfont{K}_{h,k_i}$, hence the normal form of $(\wfont{L}_h 1)^i \wfont{C}_{i}3$ is equal to $\wfont{K}_{h,k_i}$. Also, from the definition of $\wore$ and the fact that $\wfont{L}_h\wor \wfont{K}_{h,k_i}$ we have $(\wfont{L}_h1)^{j-1}\wfont{C}_j\wor \wfont{K}_{h,k_i}\sim (\wfont{L}_h1)^i\wfont{C}_i3$, for all $j$ from $1$ to $r$. Thus the normal form of  $(\wfont{L}_h)^{r-1}\wfont{C}_r0(\wfont{L}_h)^{r-2}\wfont{C}_{r-1}\ldots (\wfont{L}_h1)^{i}\wfont{C}_{i+1}0(\wfont{L}_h1)^{i-1}\wfont{C}_{i}3$ is equal to $\wfont{K}_{h,k_i}$. Hence $\textbf{u}(\wfont{A})$ satisfies the required conditions.\eprv

We will give the definition of {\it parametric relative right total interpretation}. Suppose we have signatures $\sigma_1$ and $\sigma_2$ and first-order variables $p_1,\ldots,p_n$. We consider the notion of parametric relative translation with parameters $p_1,\ldots,p_n$ from the first-order language of the signature $\sigma_1$ to the first-order language of the signature $\sigma_2$. A translation $\ffont{tr}$ of the considered type is determined by formula $\ffont{D}_{\ffont{tr}}(x,p_1,\ldots,p_n)$ of the signature $\sigma_2$ that defines the domain of translation and formulas of signature $\sigma_2$ that are interpretations of symbols from $\sigma_1$; the late formulas have additional arguments $p_1,\ldots,p_n$. We obtain the $\ffont{tr}$-translation of an arbitrary firs-order formula of the signature $\sigma_1$ as the extension of the translation of symbols from $\sigma_1$ with quantifiers relativised to $D(x,p_1,\ldots,p_n)$. Suppose $\tfont{T}_1$ is a theory of the signature $\sigma_1$, $\tfont{T}_2$ is a theory of the signature $\sigma_2$. Suppose we have a translation of the considered type $\ffont{tr}\colon\varphi \longmapsto \varphi^*$: 
$$\{\ffont{A}\mid \ffont{A}\mbox{ is a proposition of the signature $\sigma_1$ and }\tfont{T}_2\vdash \forall p_1,\ldots, p_n (\ffont{A}^*)\}\subset \tfont{T}_1.$$
Then we call $\ffont{tr}$ a parametric relative right total interpretation of $\tfont{T}_1$ in $\tfont{T}_2$ with parameters $p_1,\ldots,p_n$. 

Let us consider the case when $\tfont{T}_1$ is the elementary theory of a class of models $\sfont{B}$, $\tfont{T}_2$ is the elementary theory of a model $\mathfrak{A}$. Suppose we have a translation $\ffont{tr}\colon \varphi\longmapsto \varphi^*$ of considered type. Let us construct a family of models $\mathfrak{I}(p_1,\ldots,p_n)$ of the signature  $\sigma_1$. For $a_1,\ldots a_n\in \mathfrak{A}$ the domain of the model $\mathfrak{I}(a_1,\ldots,a_n)$ is the set $\sfont{I}(a_1,\ldots,a_n)=\{b\in \mathfrak{A} \mid \mathfrak{A} \models \ffont{D}_{\ffont{tr}}(b,a_1,\ldots,a_n)\}$. Also, for $a_1,\ldots,a_n\in \mathfrak{A}$ and $k$-ary predicate symbol $\ffont{P}$ from $\sigma_1$, the interpretation of  $\ffont{P}$ in $\mathfrak{I}(a_1,\ldots,a_n)$ is $\{ (b_1,\ldots,b_k) \mid b_1,\ldots b_k \in \sfont{I}(a_1,\ldots,a_n), \mathfrak{A}\models \ffont{P}^{*}(a_1,\ldots,a_n,b_1,\ldots,b_k)\}$; for  a $k$-ary functional symbol $f$ from $\sigma_1$ we give the interpretation of $f$ in  $\mathfrak{I}$ as the following: $$\mathfrak{I}\models f(b_1,\ldots,b_k)=c \iff \mathfrak{A}\models \ffont{F}^*(b_1,\ldots,b_k,c,a_1,\ldots,a_n),$$ where $\ffont{F}$ is the formula $f(x_1,\ldots,x_k)=y$. Thus we defined the family of models $\mathfrak{I}(p_1,\ldots,p_n)$. If, for every $\mathfrak{B}$ from $\sfont{B}$, there are $a_1,\ldots,a_n\in\mathfrak{A}$ such that $\mathfrak{I}(a_1,\ldots,a_n)$ is isomorphic to $\mathfrak{B}$, then the translation $\ffont{tr}$ is a parametric relative right total interpretation of $\tfont{T}_1$ in $\tfont{T}_2$.

We call a theory $\tfont{T}$  {\it hereditary undecidable} if every subtheory $\tfont{T}'\subset\tfont{T}$ is undecidable. 

The following well-known fact obviously holds:

\begin{fact} \label{hereditary_undecidability_fact} Suppose  $\tfont{T}_1$ is a theory with the finite signature, and $\tfont{T}_2$ is a theory such that $\tfont{T}_1$ is  hereditary undecidable, and there is a parametric relative right total interpretation of $\tfont{T}_1$ in $\tfont{T}_2$. Then the theory $\tfont{T}_2$ is undecidable. 
\end{fact} 

\begin{lemma} \label{E_undecidability} The theory $\Th(\word_3^N,\wor,\romb{1},\romb{3})$ is undecidable.
\end{lemma}

\sprv We consider the class $\sfont{L}^2_{\textit{fin}}$  of all  models $(\sfont{B},\ffont{L}_1,\ffont{L}_2)$ such that $\sfont{B}$ is a finite set, $\ffont{L}_1$ and $\ffont{L}_2$ are strict linear orders on it. The elementary theory of $\sfont{L}^2_{\textit{fin}}$ is hereditary undecidable \cite{Lav63}. 

Let us built a parametric relative right total interpretation $\ffont{tr}\colon \varphi \longmapsto \varphi^*$  of $\Th(\sfont{L}^2_{\textit{fin}})$ in $\Th(\word_3^N,\wor,\romb{1},\romb{3})$; if we will built this interpretation then by Lemma \ref{hereditary_undecidability_fact} we will prove the lemma. $p$ will be the only parameter of the interpretation. We put 
\begin{enumerate}
\item $\ffont{D}_{\ffont{tr}}(x,p)\rightleftharpoons \ffont{Sl}(p,x)$;
\item $(x_1\ffont{L}_1x_2)^*\rightleftharpoons x_1 \wor x_2$;
\item $(x_1\ffont{L}_2 x_2)^*\rightleftharpoons \romb{3}x_1 \wor \romb{3}x_2$.
\end{enumerate}

From $\ffont{tr}$ we obtain the family of models $\mathfrak{I}(p)$. Let us show that for every model  $(\sfont{B},\ffont{L}_1,\ffont{L}_2)\in \sfont{L}_{\textit{fin}}^2$ there exists  $\wfont{A}\in \word_3^N$ such that $\mathfrak{I}(\wfont{A})$ is isomorphic to $(\sfont{B},\ffont{L}_1,\ffont{L}_2)$. We put $h=|\sfont{B}|$.  We enumerate elements of  $\sfont{B}$ with respect to $\ffont{L}_1$: $b_1 \ffont{L}_1 b_2 \ffont{L}_1\ldots \ffont{L}_1 b_h$. Suppose we have: $b_{s_1}\ffont{L}_2b_{s_2}\ffont{L}_2\ldots \ffont{L}_2b_{s_h}$. By Lemma \ref{seq_existence}, there is $\wfont{A}$ such that $\textbf{u}(\wfont{A})=(\wfont{K}_{h,s_1},\wfont{K}_{h,s_2},\ldots,\wfont{K}_{h,s_h})$. Clearly, $\mathfrak{I}(\wfont{A})$ is isomorphic to $(\sfont{B},\ffont{L}_1,\ffont{L}_2)$. Therefore, $\ffont{tr}$ is the required parametric relative right total interpretation.\eprv

Using Lemma \ref{E_undecidability} and Lemma \ref{interpretability_lemma} we conclude 
\begin{theorem} For every $\alpha\in[3,\omega]$ the theory $\Th(\word_\alpha^N,\wor,\romb{1},\romb{3})$ is undecidable.
\end{theorem}
\begin{theorem}  For every $\alpha\in[3,\omega]$ the theory $\Th(\word_\alpha^N,\wor,\varLambda,\langle \romb{i}\mid i\in\omega, i\le\alpha\rangle)$ is undecidable.
\end{theorem}


\section{Some theories of ordinals and words}
\label{interpretability_section}
In this section we prove that, for $\alpha\in[2,\omega)$, theories $\Th(\word_\alpha^N,\wor,\varLambda,\romb{0},\romb{1},\romb{2})$ are undecidable. For every $\alpha\in[2,\omega)$, we will construct an interpretation of $\Th(\word_\alpha^N,\wor,\varLambda,\romb{0},\romb{1},\romb{2})$ in the weak monadic theory of $(\omega_\alpha,\ffont{R})$; here $\ffont{R}$ is a binary relation that is related to the standard cofinal sequences. The weak monadic theory of $(\omega_\alpha,\ffont{R})$ is decidable \cite{Brau09}. In order to construct this interpretation, we construct the following sequence of interpretations of structures, for all $\alpha\in[2,\omega)$:
\begin{enumerate}
\item an interpretation of $(\word_\alpha^N,=,\wor,\romb{0},\romb{1},\romb{2})$ in $\WMSExt{(\omega_\alpha,<,\psi)}$ (the structure consists of the ordinals below $\omega_\alpha$, the finite multisets of ordinals below $\omega_{\alpha}$, the standard order on ordinals, a special function $\psi$ on ordinals, and some natural predicates for work with multisets), we construct the interpretation in Lemma \ref{WMMSO_words_interpretability};
\item an interpretation of $\WMSExt{(\omega_\alpha,<,\psi)}$ in $\WSExt{(\omega_\alpha,<,\psi)}$ (the structure consists of the ordinals below $\omega_\alpha$, the finite sets of ordinals below $\omega_{\alpha}$, the standard order on ordinals, the function $\psi$, and the predicate $\in$), we construct the interpretation in  Lemma \ref{WMSO_WMMSO_interpretability};
\item an interpretation of $\WSExt{(\omega_\alpha,<,\psi)}$ in $\WSExt{(\omega_\alpha,\ffont{R})}$ (the structure consists of the ordinals below $\omega_{\alpha}$, the finite sets of ordinals, the relation $\ffont{R}$, and the predicate $\in$), we construct the interpretation in Lemma \ref{psi_interpretability}.
\end{enumerate}
Note that $\Th(\WSExt{(\omega_\alpha,\ffont{R})})$ essentially is the weak monadic theory of $(\omega_\alpha,\ffont{R})$.

There are functions $o_n\colon \NF\cap \sfont{S}_n \to \On$. We simultaneously define the functions (essentially, we recall the definition of the functions $o_n$ from \cite[Section 6]{Bek04-2}):
\begin{itemize}
\item $o_n(n^k)=k$;
\item $o_n(\wfont{A}_1n\wfont{A}_2n\ldots n\wfont{A}_k)=\omega^{o_{n+1}(\wfont{A}_1)}+\omega^{o_{n+1}(\wfont{A}_2)}+\ldots+\omega^{o_{n+1}(\wfont{A}_k)}$, where $k\ge 1$, $\wfont{A}_1,\ldots,\wfont{A}_k\in \sfont{S}_{n+1}$ and $\wfont{A}_k\wore \wfont{A}_{k-1}\wore \ldots \wore \wfont{A}_1\ne \varLambda$.
\end{itemize}

Cantor normal form of an ordinal $\alpha$ is the form $\alpha=\omega^{\alpha_1}+\ldots+\omega^{\alpha_n}$, where $\alpha_1\ge\ldots\ge\alpha_n$ and $n\ge 0$. There is the only Cantor normal form for a given ordinal.

We prove by induction on $k$ that, for every $k$ and $n$, the function $o_n$ is an isomorphism of $(\sfont{S}_n\cap \word_{n+k}^N,\wor)$ and $(\omega_{k+1},<)$.

In this section we use many-sorted predicate calculus. The models of the many-sorted predicate calculus are models with several domains, i.e. with one domain for every type of variables. The notions of elementary theory, definable predicate, definable set, and definable function can be reformulated in a natural way for the case of models of many-sorted predicate calculus.

 For every set $\sfont{A}$, we denote by  $\pfin(\sfont{A})$ the set of all finite subsets of $\sfont{A}$ . We call a function $f$ a {\it finite multiset} if the domain $\dom(f)$ is finite and the range $\ran(f)$ is included in $[1,\omega)$. Multiset $f$ is {\it included} in $g$, if $\dom(f)\subset\dom(g)$ and for all $x\in\dom(f)$ we have $f(x)\le g(x)$; in this situation we write $f\subset_Mg$. We define $\mathfrak{i}_f(x)$ the {\it multiplicity} of $x$ in a finite multiset $f$. If $x\in\dom(f)$, then we put $\mathfrak{i}_f(x)=f(x)$. Otherwise, we put $\mathfrak{i}_f(x)=0$. For every $x$ and multiset $f$ we define $x\in_M f \defiff \mathfrak{i}_f(x)>0$.  For every set $\sfont{A}$, we denote by $\multi(\sfont{A})$ the  set of all finite multisets $f$ such that all elements of $f$ are from $\sfont{A}$. 

We consider a model $\mathfrak{A}$ of one-sorted predicate calculus with the domain $\sfont{A}$. We define two  models that extends $\mathfrak{A}$ with an additional domain. The model $\WSExt{\mathfrak{A}}$ is the extension of $\mathfrak{A}$ by the additional domain $\pfin(A)$ and the predicate $\in$ on $\sfont{A}\times \pfin(\sfont{A})$. The $\WMSExt{\mathfrak{A}}$ is the extension of $\mathfrak{A}$ by the additional domain $\multi(\sfont{A})$, the predicate $\in_M$, on $\sfont{A}\times \multi(\sfont{A})$ and the predicate $\subset_M$ on $\multi(\sfont{A})\times \multi(\sfont{A})$. Note that $\Th(\WSExt{\mathfrak{A}})$ is the weak monadic theory of $\mathfrak{A}$.

Further, we will prove several lemmas about definability of several predicates in models $(\alpha,<)$, $\WSExt{(\alpha,<)}$, and $\WMSExt{(\alpha,<)}$, where $\alpha$ is an ordinal; note that we use von Neumann ordinals and hence $$\alpha=\{\beta\in \On\mid \beta<\alpha\}.$$ Obviously, all sets, predicates, and functions that are definable in $(\alpha,<)$ are also definable in $\WSExt{(\alpha,<)}$ and $\WMSExt{(\alpha,<)}$. 

\begin{lemma} \label{Uni_definitions} Suppose  $\alpha>0$ is a limit ordinal. Then the following predicates, functions, and elements are definable in the model $(\alpha,<)$: 
\begin{enumerate}
\item function $S\colon\On\to\On$, $S\colon \beta\longmapsto\beta+1$, restricted to $\alpha$;
\item element $0$;
\item predicate $x\in \Lim$, where $\Lim$ is the class of all non-zero non-successor ordinals, restricted to $\alpha$;
\item equivalence relation $\ffont{FinDif}(x,y)$, where $$\ffont{FinDif}(\beta,\gamma) \defiff \exists n\in \omega (\beta+n=\gamma\lor\beta=\gamma+n),$$  restricted to $\alpha$.
\end{enumerate}
\end{lemma}
\sprv For every $\beta,\gamma\in \alpha$ the following equivalences holds:


$1.\ S(\beta)=\gamma\iff \beta<\gamma \& \forall \delta\in\alpha (\gamma\le \delta \lor \delta\le \beta);$

$2.\ \beta=0\iff \forall \delta\in \alpha(\beta\le\delta);$

$3.\ \beta\in\Lim \iff \beta\ne 0\& \forall \delta_1\in \alpha (\delta_1 <\beta \to \exists \delta_2\in \alpha (\delta_2<\beta \& \delta_1<\delta_2));$

$\begin{aligned}4.\ \ffont{FinDif}(\beta,\gamma)\iff  \beta=\gamma & \lor (\beta<\gamma \& \forall \delta\in \alpha( \beta<\delta\le\gamma \to \delta\not \in \Lim)) \lor\\ & (\gamma<\beta \& \forall \delta\in \alpha( \gamma<\delta\le\beta \to \delta\not \in \Lim)).\\ \end{aligned}$

The equivalences show that the functions, predicates, and elements under considerations are definable.\eprv

We denote by $\emptyset^{\mathrm{M}}$ the empty multiset. 

\begin{lemma} Suppose $\alpha\in \On$. Then  the function $$\min\colon \pfin(\alpha)\setminus\{\emptyset\}\to \alpha$$ is definable in $\WSExt{(\alpha,<)}$ and the function $$\min\colon \multi(\alpha)\setminus\{\emptyset^{\mathrm{M}}\}\to \alpha$$  is definable in $\WMSExt{(\alpha,<)}$.\end{lemma}
\sprv For every $\sfont{Q}\in \pfin(\alpha)\setminus\emptyset$ and $\beta\in \alpha$, we have the following equivalence 
$$\min(\sfont{Q})=\beta\iff \beta\in \sfont{Q} \& \forall \gamma\in \alpha(\gamma<\beta\to\gamma\not\in \sfont{Q}).$$
We have built the required definition in $\WSExt{(\alpha,<)}$. Similarly, we construct the required definition in $\WMSExt{(\alpha,<)}$.\eprv

\begin{lemma} \label{Multi_defenitions} Suppose $\mathfrak{A}$ is a one-sorted model with the domain $\sfont{A}$. Then the following predicates are definable in the model $\WMSExt{\mathfrak{A}}$:
\begin{enumerate}
\item the predicate $\ffont{CLess}(x,\sfont{X},\sfont{Y})$ such that for all  $(a,\sfont{Q}_1,\sfont{Q}_2)\in \sfont{A}\times \multi(\sfont{A})\times \multi(\sfont{A})$  we have $\ffont{CLess}(a,\sfont{Q}_1,\sfont{Q}_2)$ iff the multiplicity of $a$ in $\sfont{Q}_1$ is less than the multiplicity of $a$ in $\sfont{Q}_2$;
\item the predicate $\ffont{CEq}(x,\sfont{X},\sfont{Y})$ such that for all $(a,\sfont{Q}_1,\sfont{Q}_2)\in \sfont{A}\times \multi(\sfont{A})\times \multi(\sfont{A})$ we have $\ffont{CEq}(a,\sfont{Q}_1,\sfont{Q}_2)$ iff the multiplicity of $a$ in $\sfont{Q}_1$ is equal to the multiplicity of $a$ in $\sfont{Q}_2$;
\item predicate $\ffont{CS}(x,\sfont{X},\sfont{Y})$ such that for all $(a,\sfont{Q}_1,\sfont{Q}_2)\in \sfont{A}\times \multi(\sfont{A})\times \multi(\sfont{A})$ we have $\ffont{CS}(a,\sfont{Q}_1,\sfont{Q}_2)$ iff the multiplicity of $a$ in $\sfont{Q}_1$ is equal to the multiplicity of $a$ in $\sfont{Q}_2$ minus $1$.
\end{enumerate}
\end{lemma}
\sprv Obviously, for all triples $(a,\sfont{Q}_1,\sfont{Q}_2)\in A\times \multi(A)\times \multi(A)$, the following equivalences holds:
\begin{enumerate}
\item $\ffont{CLess}(a,\sfont{Q}_1,\sfont{Q}_2)\iff \exists \sfont{Q}_3\in\multi(A) (\forall b\in A(b\in_M \sfont{Q}_3\leftrightarrow a=b)\& \sfont{Q}_3\subset_M\sfont{Q}_2\&\lnot \sfont{Q}_3\subset_M \sfont{Q}_1).$
\item $\ffont{CEq}(a,\sfont{Q}_1,\sfont{Q}_2)\iff \lnot \ffont{CLess}(a,\sfont{Q}_1,\sfont{Q}_2)\& \lnot \ffont{CLess}(a,\sfont{Q}_2,\sfont{Q}_1).$
\item $\ffont{CS}(a,\sfont{Q}_1,\sfont{Q}_2)\iff$
 
$\;\;\ffont{CLess}(a,\sfont{Q}_1,\sfont{Q}_2)\& \forall \sfont{Q}_3 \in \multi(A) (\lnot (\ffont{CLess}(a,\sfont{Q}_3,\sfont{Q}_2)\&  \ffont{CLess}(a,\sfont{Q}_1,\sfont{Q}_3))).$
\end{enumerate}

Therefore, the required predicates are definable.\eprv

We define function $\psi\colon \On\to \On$: 
\begin{itemize}
\item $\psi(0)=\omega$;
\item $\psi(\omega^{\alpha_1}+\omega^{\alpha_2}+\ldots+\omega^{\alpha_n})=\omega^{\alpha_1}+\omega^{\alpha_2}+\ldots+\omega^{\alpha_{n-1}}+\omega^{\alpha_n+1}$, where $n\ge 1$ and $\alpha_1\ge\alpha_2\ge\ldots\ge\alpha_n$.
\end{itemize}

We say that an ordinal $\alpha$  is {\it closed under $\psi$}, if for every $\beta<\alpha$ we have $\psi(\beta)<\alpha$.

\begin{remark} \label{psi_closed} An ordinal $\alpha$ is closed under $\psi$ iff either  $\alpha=0$ or $\alpha=\omega^{\omega\cdot\beta}$, for some $\beta>0$.
\end{remark}

Below in several lemmas we construct interpretations of some individual many-sorted models in other individual many-sorted models. We construct an interpretation of a many-sorted model $\mfont{A}$ in a many-sorted model $\mfont{B}$ by  
\begin{enumerate}
\item a choice of the corresponding type of $\mfont{B}$, for every type of $\mfont{A}$;
\item a choice of injective functions $f_i\colon x\longmapsto x^I$ from domains of $\mfont{A}$ to the corresponding domains of $\mfont{B}$;
\item a choice of formulas $\ffont{D}_i(x)$ in the language of $\mfont{B}$ that defines the full images under $f_i$ of the corresponding domains of $\mfont{A}$;
\item for  all predicates and functions from the signature of $\mfont{A}$, a choice of a formula that defines in $\mfont{B}$ the image under functions $f_i$ of this predicate or function.
\end{enumerate}

\begin{lemma} \label{WMMSO_words_interpretability}Suppose $\alpha$ is an ordinal from $2$ to $\omega$. Then the model $(\word^N_\alpha,\wor,\varLambda,\romb{0},\romb{1},\romb{2})$ is interpretable in $\WMSExt{(\omega_\alpha,<,\psi)}$.
\end{lemma} 
\sprv We note two facts. From Remark \ref{psi_closed} it follows that $\omega_\alpha$ is closed under $\psi$. The function $o_1$ is a bijection from $\word_\alpha^N\cap \sfont{S}_1$ to $\omega_\alpha$.

We consider a word $\wfont{A}\in \word^N_\alpha$ and give it's interpretation $\wfont{A}^I$. We can represent in the unique way the word $\wfont{A}$ in the form $\wfont{A}_10\ldots 0\wfont{A}_n$, where $n\ge 0$ and $\wfont{A}_1,\ldots,\wfont{A}_n\in \word_\alpha^N\cap \sfont{S}_1$. We put the multiplicity of $\gamma\in \omega_\alpha$ in $\wfont{A}^I$ to be equal to the number of $i$ from $1$ to $n$ such that $o_1(\wfont{A}_i)=\gamma$. Obviously, we have defined a bijection $\wfont{A}\mapsto \wfont{A}^I$ from $\word_\alpha^N$ to $\multi(\omega_\alpha)$.

We define a predicate $\wor^I$ the interpretation of the  predicate $\wor$:
$$\sfont{X}\wor^I\sfont{Y} \rightleftharpoons \exists x (\ffont{CLess}(x,\sfont{X},\sfont{Y})\&\forall y>x (\ffont{CEq}(y,\sfont{X},\sfont{Y}))).$$

Let us prove that for words $\wfont{A},\wfont{B}\in\word^N_{\alpha}$ we have 
 $$\wfont{A}\wor \wfont{B} \iff \wfont{A}^I \wor^I \wfont{B}^I.$$ 
We find $\wfont{A}_1,\ldots,\wfont{A}_n,\wfont{B}_1,\ldots,\wfont{B}_m\in \sfont{S}_1$ such that  $\wfont{A}$ is equal to $\wfont{A}_10\wfont{A}_20\ldots0\wfont{A}_n$ and $\wfont{B}$ is equal to $\wfont{B}_10\wfont{B}_20\ldots 0\wfont{B}_m$. We denote by $\sfont{A}$  the interpretation $\wfont{A}^I$ and we denote by $\sfont{B}$ the interpretation $\wfont{B}^I$. Let us prove that
 $$\wfont{A}_10\wfont{A}_20\ldots0\wfont{A}_n\wor \wfont{B}_10\wfont{B}_20\ldots 0\wfont{B}_m \iff \WMSExt{(\omega_\alpha,<,\psi)}\models \sfont{A} \wor^I \sfont{B}.$$

Suppose we have $\WMSExt{(\omega_\alpha,<,\psi)}\models \sfont{A}\wor^I \sfont{B}$. Let us prove that $\wfont{A}_10\wfont{A}_20\ldots0\wfont{A}_n\wor \wfont{B}_10\wfont{B}_20\ldots 0\wfont{B}_m$. There exists an ordinal $\gamma$ such that the multiplicity of  $\gamma$ in $\sfont{A}$ is less than the multiplicity of $\gamma$ in $\sfont{B}$ and for all $\delta\in(\gamma,\omega_\alpha)$ the multiplicity of $\delta$ in $\sfont{A}$ and  the multiplicity of $\delta$ in $\sfont{B}$ are equal. Suppose the multiplicity of $\gamma$ in $\sfont{A}$ is equal to $l$. Suppose $k$ is the $(l+1)$-th element of $\{i\mid o_1(\wfont{B}_{i})=\gamma\}$ in the sense of standard ordering of natural numbers; note that from definition of $\wor^I$ it follows that we can find such a number $k$. Then from the definition of $\NF$ it follows that for all $i$ from $1$ to $k-1$ we have $\wfont{A}_i=\wfont{B}_i$. We have either $n=k-1$ or $\wfont{A}_k\wor\wfont{B}_k$. Thus the sequence $(\wfont{A}_1,\ldots,\wfont{A}_n)$ is lexicographically less than  $(\wfont{B}_1,\ldots,\wfont{B}_m)$ and $\wfont{A}_10\ldots0\wfont{A}_n\wor \wfont{B}_10\ldots0\wfont{B}_m$. 

Now we assume that  $\wfont{A}_10\ldots0\wfont{A}_n\wor \wfont{B}_10\ldots0\wfont{B}_m$. Let us show that $\WMSExt{(\omega_\alpha,<,\psi)}\models \sfont{A}\wor^I \sfont{B}$. From the definitions of $\wore$ and $\NF$ it follows that there exists $k$ such that for all $i$ from $1$ to $k-1$ we have $\wfont{A}_i=\wfont{B}_i$ and either $n=k-1$ or $\wfont{A}_k\wor \wfont{B}_k$. From the late it follows that for all $i$ from $k$ to $n$ we have $\wfont{A}_i\wor \wfont{B}_k$, and hence $o_1(\wfont{A}_i)<o_1(\wfont{B}_k)$. We take $o_1(\wfont{B}_k)$ as  $x$ from the definition of $\wor^I$, hence $\WMSExt{(\omega_\alpha,<)}\models \sfont{A}\wor^I \sfont{B}$. Thus $\wor^I$ is the interpretation of $\wor$.

The function $\romb{0}$ and the element $\varLambda$ is definable in $(\word^N_\alpha,\wor)$ and we obtain  interpretations of $\romb{0}$ and $\varLambda$ for free.

Note that for a word $\wfont{A}\in \NF\cap \sfont{S}_1$ we have  $\psi(o_1(\wfont{A}))=o_1(\romb{2}\wfont{A})$.

We define the functions $\romb{1}^I$ and $\romb{2}^I$ that will be the interpretations of $\romb{1}$ and $\romb{2}$, respectively:
$$\begin{aligned}\romb{1}^I \sfont{X}=\sfont{Y} \rightleftharpoons (\sfont{X}=\emptyset^M & \to \sfont{Y}=\emptyset^M) \& (\sfont{X}\ne \emptyset^M \to \\ &\forall x>S(\min(\sfont{X}))(\ffont{CEq}(x,\sfont{X},\sfont{Y})) \& \\ & \ffont{CS}(S(\min(\sfont{X})),\sfont{X},\sfont{Y})\&\\ &\forall x<S(\min(\sfont{X}))(\lnot x\in_M \sfont{Y})),\\\end{aligned}$$
$$\begin{aligned}\romb{2}^I \sfont{X}=\sfont{Y} \rightleftharpoons (\sfont{X}=\emptyset^M & \to \sfont{Y}=\emptyset^M) \& (\sfont{X}\ne \emptyset^M \to \\ &\forall x>\psi(\min(\sfont{X}))(\ffont{CEq}(x,\sfont{X},\sfont{Y})) \& \\ & \ffont{CS}(\psi(\min(\sfont{X})),\sfont{X},\sfont{Y})\& \\ & \forall x<\psi(\min(\sfont{X}))(\lnot x\in_M \sfont{Y})).\\\end{aligned}$$
We claim that $\romb{2}^I$ is an interpretation of $\romb{2}$; we omit the proof of the fact that $\romb{1}^I$ is an interpretation of $\romb{1}$, because it is similar to our claim. We consider a word $\wfont{A}\in \word_\alpha^N$ of the form $\wfont{A}_10\ldots0\wfont{A}_n$, where $\wfont{A}_1,\ldots,\wfont{A}_n\in \word_\alpha^N\cap \sfont{S}_1$. From Lemma \ref{maximal_lexicographic_algorithm} it follows that the lexicographically maximal subsequence of the sequence $(\wfont{A}_1,\ldots,\wfont{A}_{n-1},\wfont{A}_n2)$ is of the form $(\wfont{A}_1,\ldots,\wfont{A}_k,\wfont{A}_n2)$, where $k\in \{0,\ldots,n-1\}$. And, for all $i$ from $1$ to $k$, we have $\wfont{A}_n2\wore \wfont{A}_i$, hence, for all $i$ from $k+1$ to $n-1$, we have $\wfont{A}_i\wor \wfont{A}_n2$. Therefore $\romb{2}^I$ is an interpretation of $\romb{2}$.\eprv

\begin{lemma} \label{WMSO_WMMSO_interpretability}Suppose an ordinal $\alpha$  is closed under $\psi$. Then $\WMSExt{(\alpha,<,\psi)}$ is interpretable in $\WSExt{(\alpha,<,\psi)}$.
\end{lemma}
\sprv We will interpret an ordinal $\beta\in\alpha$ by the ordinal $\beta^I=\omega\cdot \beta$. Clearly, we have define an injection of $\alpha$ into itself. For a given set $\sfont{A}\in \multi( \alpha)$  we build  $\sfont{A}^I\in \pfin(\alpha)$ the interpretation of $\sfont{A}$. Suppose $\beta_1,\ldots,\beta_n$ are pairwise different ordinals below $\alpha$ such that every ordinal that have non-zero multiplicity in $\sfont{A}$ is some $\beta_i$. We denote by $k_1,\ldots,k_n\in \omega$ the multiplicities of the ordinals $\beta_1,\ldots,\beta_n$ in $\sfont{A}$, respectively. We put $\sfont{A}^I=\{\omega\cdot \beta_i+(k_i-1)\mid i\in \{1,\ldots,n\}\}$. 

Note that the mapping $(\beta,k)\longmapsto \omega \cdot \beta + (k-1)$ is a bijection between $\alpha\times (\omega\setminus \{0\})$ and $\alpha$. Thus the mapping $\sfont{A}\longmapsto \sfont{A}^I$ is an injection. Let us show that the set  $\sfont{U}_1$ of all interpretations of ordinals is definable in $\WSExt{(\alpha,<,\psi)}$. Really, for every $\beta\in \alpha$ 
$$\beta\in \sfont{U}_1\iff \beta \in \Lim\lor \beta=0.$$
Now we show that the set $\sfont{U}_2$ of all interpretations of multisets is definable $\WSExt{(\alpha,<,\psi)}$. For every $\sfont{A}\in \pfin(\alpha)$ we have
$$\sfont{A}\in \sfont{U}_2\iff \forall \beta \in \alpha \forall \gamma \in \alpha ((\beta \in \sfont{X} \& \gamma \in \sfont{X})\to (\ffont{FinDif}(\beta,\gamma)\leftrightarrow \beta=\gamma)).$$

We give $\in_M^I$, $\subset_M^I$, $<^I$, $\psi^I$ the interpretations of $\in_M$, $\subset_M$, $<$, $\psi$, correspondingly.

$$x\in_M^I \sfont{X}\rightleftharpoons x\in \sfont{U}_1 \& \sfont{X}\in \sfont{U}_2 \& \exists y \in \sfont{X} (\ffont{FinDif}(x,y));$$
$$x<^Iy\rightleftharpoons x\in \sfont{U}_1\&y\in \sfont{U}_1\&x<y;$$
$$\psi^I(x)=y\rightleftharpoons x\in \sfont{U}_1\&y\in \sfont{U}_1\&(x=0\to y=\psi(\psi(0)))\& (x\ne 0 \to y=\psi(x));$$
$$\sfont{X}\subset_M^I \sfont{Y}\rightleftharpoons \sfont{X},\sfont{Y}\in \sfont{U}_2\&\forall x\in \sfont{X} \exists y\in \sfont{Y} (\ffont{FinDif}(x,y)\& x\le y).$$

Clearly, the definitions give us the required interpretation.\eprv

There is the standard choice of cofinal sequences for ordinals less than $\varepsilon_0$. For every ordinal $\alpha\in\Lim$ with the Cantor normal form $\omega^{\alpha_1}+\ldots+\omega^{\alpha_k}$,  $\alpha[n]$ the $n$-th member of the standard cofinal sequence for $\alpha$ is given as following:
\begin{enumerate} 
\item $\alpha[n]=\omega^{\alpha_1}+\ldots+\omega^{\alpha_{k-1}}+\omega^{\beta}(n+1)$ if $\alpha_k\not\in\Lim$ and $\alpha_k=\beta+1$;
\item $\alpha[n]=\omega^{\alpha_1}+\ldots+\omega^{\alpha_{k-1}}+\omega^{\alpha_k[n]}$ if $\alpha_k\in\Lim$.
\end{enumerate}
With the use of cofinal sequences we define the relation $\ffont{R}$ on ordinals less than $\varepsilon_0$:
$$\alpha \ffont{R}\beta\defiff \beta=\alpha+1\lor \exists n\in \omega (\alpha=\beta[n]).$$
Clearly, the transitive closure of $\ffont{R}$ is the standard order on ordinals $<$.

Laurent~Braud \cite{Brau09} have proved the following theorem:
\begin{theorem} \label{Braud_theorem}For all $\alpha\in[1,\varepsilon_0)$, the theory  $\Th(\WSExt{(\alpha,\ffont{R})})$ is decidable.
\end{theorem}

\begin{lemma}\label{psi_interpretability} The model $\WSExt{(\alpha,<,\psi)}$ is interpretable in the model $\WSExt{(\alpha,\ffont{R})}$. 
\end{lemma}
\begin{proof} We only need to show that $\psi$ is definable in $\WSExt{(\omega_\alpha,<,\ffont{R})}$. Suppose $\beta\in \omega_\alpha$ is a non-zero ordinal. Let us show that $\psi(\beta)$ is the second ordinal  $\gamma$ such that $\beta \ffont{R}\gamma$. Suppose the Cantor normal form  of $\beta$ is $\omega^{\beta_1}+\ldots +\omega^{\beta_{k-1}}+\underbrace{\omega^{\beta_k}+\ldots+\omega^{\beta_k}}_{\mbox{$n$ times}}$, where $k,n\ge 1$ and $\beta_1\ge\beta_2\ge\ldots\ge\beta_{k-1}>\beta_k$. Clearly, $\psi(\beta)= \omega^{\beta_1}+\ldots +\omega^{\beta_{k-1}}+\omega^{\beta_{k}+1}$. Obviously, $\beta \ffont{R}\psi(\beta)$ and $\beta \ffont{R}(\beta+1)$. Let us prove by a contradiction that for all $\gamma\in(\beta+1,\psi(\beta))$ we don't have $\beta \ffont{R}\gamma$. Suppose $\gamma\in(\beta+1,\psi(\beta))$ and $\beta \ffont{R}\gamma$. Then the Cantor normal form of $\gamma$ is $\omega^{\beta_1}+\ldots +\omega^{\beta_{k-1}}+\underbrace{\omega^{\beta_k}+\ldots+\omega^{\beta_k}}_{\mbox{$n$ times}}+\omega^{\gamma_1}+\ldots+\omega^{\gamma_s}$, where  $s\ge 1$ and $\beta_k>\gamma_1$. From the definition of $\ffont{R}$ it follows that $\gamma\in\Lim$ and for some $n$ we have $\gamma[n]=\beta$. But $\gamma[n]=\omega^{\beta_1}+\ldots +\omega^{\beta_{k-1}}+\underbrace{\omega^{\beta_k}+\ldots+\omega^{\beta_k}}_{\mbox{$n$ times}}+\omega^{\gamma_1}+\ldots+\omega^{\gamma_{s-1}}+(\omega^{\gamma_s})[n]$ and $(\omega^{\gamma_s})[n]\ne 0$. Thus $\gamma[n]>\beta$. The late contradicts $\beta \ffont{R}\gamma$. Hence $\psi(\beta)$ is really the second  $\gamma$ such that $\beta \ffont{R}\gamma$.

From the previous paragraph it follows that, for all $\beta,\gamma<\omega_\alpha$, we have $\psi(\beta)=\gamma$ iff
$$\begin{aligned}
(\beta=0\to \gamma\in \Lim \& \forall \delta<\gamma (\delta\not\in \Lim))\&\\
(\beta\ne 0\to \beta \ffont{R}\gamma \& \exists! \delta<\gamma(\beta \ffont{R}\delta))\\
\end{aligned}$$
Hence the function $\psi$ is definable in $\WSExt{(\omega_\alpha,<,\ffont{R})}$.\end{proof}

 Using Lemmas \ref{WMMSO_words_interpretability}, \ref{WMSO_WMMSO_interpretability}, \ref{psi_interpretability}, and Theorem \ref{Braud_theorem} we conclude that the following theorem holds:
\begin{theorem} \label{weak_decidability_theorem}For all $\alpha\in[2,\omega)$, the theory $\Th(\word^N_\alpha,\wor,\varLambda,\romb{0},\romb{1},\romb{2})$ is decidable.
\end{theorem}


\section{Elementary equivalence of some models}
\label{elementary_equivalence_section}
In the section we show that  $(\word^N_\omega,\varLambda,\wor,\romb{0},\romb{1},\romb{2})$ and $(\word^N_3,\varLambda,\wor,\romb{0},\romb{1},\romb{2})$ are elementary equivalent. Thus we show that $\Th(\word^N_\omega,\varLambda,\wor,\romb{0},\romb{1},\romb{2})$ is decidable. We give a stronger form of Theorem \ref{weak_decidability_theorem}. Here we use the classical result by A.~Ehrenfeucht about elementary equivalency \cite{Ehr61} .

\begin{remark} \label{set-theoretic_formalism_remark}Further, we consider the notions of  ordinals, pairs, functions, and sequences in the set-theoretic fashion. We use von Neuman ordinals $\alpha=\{\beta\mid \beta<\alpha\}$. We use the Kuratowski definition of ordered pair $(x,y)=\{\{x\},\{x,y\}\}$. We consider functions $f$ as the set of pairs $\{(x,f(x))\mid x\in\dom(f)\}$. We consider sequences $\langle a_{\beta}\mid \beta<\alpha\rangle$ as the functions $\{(\beta,a_{\beta})\mid \beta<\alpha\}$. 
\end{remark} 

Suppose $\mfont{A}$ is a structure without functional symbols in the signature. We define model $\HFExt{\mfont{A}}$ with the signature that extends the signature of $\mfont{A}$ by the binary predicate symbol $\in$ and the unary predicate symbol $\ffont{At}$. The domain of the model $\HFExt{\mfont{A}}$ is the set $\sfont{A}^+$. The set $\sfont{A}^+$ is the minimal set such that $\sfont{A}\times \{\omega\}\subset \sfont{A}^+$ and $\pfin(\sfont{A}^+)\subset \sfont{A}^+$. Obviously, $\sfont{A}^{+}$ exists and unique. We define standard embedding $\pi_{\sfont{A}}\colon \sfont{A} \to \sfont{A}^+$, for every  $a\in \sfont{A}$, we put $\pi_{\sfont{A}}(a)=(a,\omega)$. Note that $a^+\in \sfont{A}^+$ is of the form $(x,\omega)$ iff $a^+\in\pi_{\sfont{A}}[\sfont{A}]$. Interpretations of a predicate symbol $\ffont{P}(x_1,\ldots,x_n)$ from the signature of $\mfont{A}$ in the model $\HFExt{\mfont{A}}$ is the following:
$$\HFExt{\mfont{A}}\models \ffont{P}(a^+_1,\ldots,a^+_n)\defiff a_1^+,\ldots,a_n^+\in \pi_{\sfont{A}}[\sfont{A}] \mbox{ and } \mfont{A}\models \ffont{P}(\pi_{\sfont{A}}^{-1}(a_1^+),\ldots,\pi_{\sfont{A}}^{-1}(a_n^+)).$$
For every $a^+\in \sfont{A}^+$ 
 $$\mfont{A}^+\models \ffont{At}(a^+)\defiff a^+\in\pi_{\sfont{A}}[\sfont{A}].$$
For all $a_1^+,a_2^+\in \sfont{A}^+$
$$\HFExt{\mfont{A}}\models a_1^+\in a_2^+\defiff a_2^+\not\in\pi_{\sfont{A}}[\sfont{A}]\mbox{ and }a_1^+\in a_2^+.$$
We have defined the model $\HFExt{\mfont{A}}$.
If $a^+\in\HFExt{\mfont{A}}$ such that $\HFExt{\mfont{A}}\models \ffont{At}(a^+)$, then we call $a^+\in\HFExt{\mfont{A}}$ an atom.

We define the notion of {\it $\omega$-tail} of an ordinal $\alpha$ with the Cantor normal form $\omega^{\alpha_1}+\omega^{\alpha_2}+\ldots+\omega^{\alpha_n}$. If $\alpha<\omega^\omega$, then $\omega$-tail of $\alpha$ is equal to $\alpha$. If $\alpha\ge\omega^\omega$, then the $\omega$-tail of $\alpha$ is the ordinal $\omega^\omega+\omega^{\alpha_{k}}+\ldots+\omega^{\alpha_n}$, where $k$ is the minimal number such that $\alpha_i<\omega$, for all $i$ from $k$ to $n$. 

In \cite{Ehr61} A.~Ehrenfeucht have proved that models $\HFExt{(\alpha_1,<)}$ and $\HFExt{(\alpha_2,<)}$ are elementary equivalent, for $\alpha_1$ and $\alpha_2$ with the same $\omega$-tail. Note that  for all $\alpha\in[2,\omega]$ the $\omega$-tails of $\omega_\alpha$ are the same. 

\begin{lemma} \label{RFOT_WMSO_interpretability} For an ordinal $\alpha\in \Lim$ the model $\WSExt{(\omega^\alpha,<,\psi)}$ is interpretable in $\HFExt{(\alpha,<)}$.
\end{lemma}
\sprv Clearly,  all axioms of $\ZF$, but Infinity Axiom and Extensionality Axiom, holds in $\HFExt{(\alpha,<)}$. A natural modification of Extensionality Axiom holds in $\HFExt{(\alpha,<)}$ 
 $$\forall x,y(\lnot \ffont{At}(x)\&  \lnot \ffont{At}(y) \& \forall z (z\in x\leftrightarrow z\in y) \to x=y).$$
 
Thus in $\HFExt{(\alpha,<)}$ we can formalize the notions from Remark \ref{set-theoretic_formalism_remark}.

Suppose $\beta\in \omega^\alpha$ and  $\omega^{\beta_1}+\ldots+\omega^{\beta_n}$ is the Cantor normal form of $\beta$. Then  we put $\beta^I=(\pi_\alpha(\beta_1),\ldots,\pi_\alpha(\beta_n))$. In $\HFExt{(\alpha,<)}$ the set of all interpretations of ordinals is definable as the set of all monotone non-decreasing sequences of atoms. For a set $\sfont{A}\in\pfin(\omega^\alpha)$, the interpretation of $\sfont{A}$ is $\sfont{A}^I=\{\beta^I\mid \beta\in \sfont{A}\}$. Obviously, the set of all interpretations of sets is definable in $\HFExt{(\alpha,<)}$. The predicate $\in$ is interpretable in a natural way. We define $<^I$ the interpretation of $<$ as the lexicographic order on monotone non-decreasing sequences of atoms. Let us define function $\psi^I$ the interpretation of $\psi$. $\psi^I((\pi_\alpha(\beta_1),\ldots,\pi_\alpha(\beta_n)))$  is equal to lexicographically minimal sequence that ends with $\pi_\alpha(\beta_n+1)$ and is lexicographically greater than $(\pi_\alpha(\beta_1),\ldots,\pi_\alpha(\beta_n))$.\eprv

Note that the translations that can be extracted from the proofs of Lemmas \ref{WMMSO_words_interpretability}, \ref{WMSO_WMMSO_interpretability}, and \ref{RFOT_WMSO_interpretability} are independent of parameters of pairs of structures. Hence from Lemma \ref{RFOT_WMSO_interpretability} it follows that the following corollaries holds:

\begin{corollary} For all $\alpha_1,\alpha_2\in [3,\omega]$, the models $\WSExt{(\omega_{\alpha_1},<,\psi)}$ and $\WSExt{(\omega_{\alpha_2},<,\psi)}$ are elementary equivalent. 
\end{corollary}

\begin{corollary}  For all $\alpha_1,\alpha_2\in [3,\omega]$, the models $\WMSExt{(\omega_{\alpha_1},<,\psi)}$ and $\WMSExt{(\alpha_2,<,\psi)}$ are elementary equivalent. 
\end{corollary}

\begin{corollary} \label{word_elementary_equivalence} Suppose  $\alpha\in[3,\omega]$. Then the models $(\word^N_\alpha,\varLambda,\wor,\romb{0},\romb{1},\romb{2})$ and $(\word^N_3,\varLambda,\wor,\romb{0},\romb{1},\romb{2})$ are elementary equivalent. 
\end{corollary}

From Corollary \ref{word_elementary_equivalence} and Theorem \ref{weak_decidability_theorem} we obtain the following stronger version of Theorem \ref{weak_decidability_theorem}:

\begin{theorem} \label{decidability_theorem}For all $\alpha\in[2,\omega]$, the theory $\Th(\word^N_\alpha,\varLambda,\wor,\romb{0},\romb{1},\romb{2})$ is decidable.
\end{theorem}


\bibliographystyle{plain}
\bibliography{bibliography_en}
\end{document}